\documentclass[a4paper,11pt]{amsart}
\usepackage{amsthm,amsfonts,amsbsy,amssymb,amsmath,amscd}
\usepackage[all]{xy}
\usepackage[colorlinks=true]{hyperref}
\usepackage{pgf, tikz}
\theoremstyle{plain}
\newtheorem{theorem}{Theorem}[section]
\newtheorem{lemma}[theorem]{Lemma}
\newtheorem{coro}[theorem]{Corollary}
\newtheorem{prop}[theorem]{Proposition}
\newtheorem{defi}[theorem]{Definition}

\newtheorem*{ques*}{Question}	
\newtheorem*{conj*}{Conjecture}

\theoremstyle{remark}
\newtheorem{remark}[theorem]{Remark}

\numberwithin{equation}{section}

\newcommand{\QQ}{\mathbb{Q}}

\newcommand{\CC}{\mathbb{C}}
\newcommand{\PP}{\mathbb{P}}
\newcommand{\ZZ}{\mathbb{Z}}

\newcommand{\CE}{\mathcal{E}}
\newcommand{\CH}{\mathcal{H}}
\newcommand{\CK}{\mathcal{K}}

\newcommand{\CO}{\mathcal{O}}

\newcommand{\CV}{\mathcal{V}}

\newcommand{\Pic}{{\mathrm{Pic}}}
\newcommand{\baselocus}{{\mathrm{Bs}}}
\newcommand{\movable}{{\mathrm{Mov}}}
\newcommand{\coeff}{{\mathrm{coeff}}}
\newcommand{\supp}{{\mathrm{Supp}}}
\newcommand{\Sym}{{\mathrm{Sym}}}

\newcommand{\vol}{{\mathrm{vol}}}

\newcommand{\rc}{{\mathrm{can}}}

\newcommand{\rounddown}[1]{{\lfloor #1 \rfloor}}

\begin{document}

\title[Noether-Severi inequality and equality]{Noether-Severi inequality and equality for irregular threefolds of general type}

\dedicatory{Dedicated to Professor Zhijie Chen on the occasion of his 80th birthday}

\author{Yong Hu}
\author{Tong Zhang}
\date{\today}

\address[Y.H.]{School of Mathematical Sciences, Shanghai Jiao Tong University, 800 Dongchuan Road, Shanghai 200240, People's Republic of China}
\email{yonghu@sjtu.edu.cn}

\address[T.Z.]{School of Mathematical Sciences, Shanghai Key Laboratory of PMMP, East China Normal University, 500 Dongchuan Road, Shanghai 200241, People's Republic of China}
\email{tzhang@math.ecnu.edu.cn, mathtzhang@gmail.com}

\begin{abstract}
We prove the optimal Noether-Severi inequality that $\vol(X) \ge \frac{4}{3} \chi(\omega_{X})$ for all smooth and irregular $3$-folds $X$ of general type over $\CC$. %This answers an open question of Z. Jiang in dimension three. 
For those $3$-folds $X$ attaining the equality, we completely describe their canonical models and show that the topological fundamental group $\pi_1(X) \simeq \ZZ^2$. As a corollary, we obtain for the same $X$ another optimal inequality that $\vol(X) \ge \frac{4}{3}h^0_a(X, K_X)$ where $h^0_a(X, K_X)$ stands for the continuous rank of $K_X$, and we show that $X$ attains this equality if and only if $\vol(X) = \frac{4}{3}\chi(\omega_{X})$.
\end{abstract}

\maketitle

%\tableofcontents

%\subjclass[2010]{primary 14J27, 14E30; secondary 14J30, 32Q57} 

%   14F45   	Topological properties
% 	32Q55   	Topological aspects of complex manifolds
% 	32Q15   	Kähler manifolds
% 	14F17   	Vanishing theorems
% 	14J30   	$3$-folds  
% 	32Q57   	Classification theorems
%   14E30       Minimal model program
%  	14J27   	Elliptic surfaces, elliptic or Calabi-Yau fibrations

%\keywords{} 

\section{Introduction}

Throughout this paper, we work over the complex numbers $\CC$, and all varieties are projective.

\subsection{Motivation}
A major problem in algebraic geometry is to classify algebraic varieties. Concerning this problem for varieties of general type, one well-known strategy is  the geographical approach. That is, to find explicit relations among birational invariants of varieties first, and then to establish the fine classification based on these relations. 

Such a strategy has been proved to be a great success for algebraic surfaces. For smooth surfaces $S$ of general type, Noether \cite{Noether} proved that
\begin{equation} \label{eq: Noether2}
    \vol(S) \ge 2p_g(S) - 4,
\end{equation}
which is now known as the Noether inequality. The classification of surfaces attaining the equality was sketched in Enriques' book \cite{Enriques} and accomplished in detail afterwards by Horikawa \cite{Horikawa1}. However, the three dimensional analogue of the problem seems to be more challenging. By a recent result of J. Chen, M. Chen and C. Jiang \cite{Chen_Chen_Jiang}, the Noether inequality
\begin{equation} \label{eq: Noether3}
    \vol(X) \ge \frac{4}{3} p_g(X) - \frac{10}{3}
\end{equation}
holds for ``almost all" smooth $3$-folds $X$ of general type, but it is not known by far if \eqref{eq: Noether3} holds in general, and the classification of $3$-folds attaining the equality is still  open \cite[Question 1.4, 1.5]{Chen_Chen_Jiang}.

For an irregular variety $X$, the inequality between $\vol(X)$ and $\chi(\omega_{X})$ attracts more attention. A notable feature is that the ratio $\frac{\vol(X)}{\chi(\omega_{X})}$ is invariant under \'etale covers, thus it carries many geometric properties of $X$ itself. This type of inequalities plays the role of the Noether inequality as in the regular case, and is often referred as Severi type inequalities in the literature \cite{Severi,Pardini,Zhang,Barja,Lu_Zuo,Jiang}. Over the last decades, a number of fundamental Severi type inequalities have been established. For example, if $X$ is of maximal Albanese dimension, then the optimal inequality 
\begin{equation} \label{eq: Severi}
	\vol(X) \ge 2(\dim X)! \chi(\omega_{X})
\end{equation}
holds. This inequality was originally stated for surfaces by Severi \cite{Severi} and
finally proved by Pardini \cite{Pardini}. Later, it was generalized to arbitrary dimension by Barja \cite{Barja} and independently by the second named author \cite{Zhang}. Under some extra assumptions on $X$, this inequality has been improved by Lu and Zuo \cite{Lu_Zuo} for surfaces, and later by Barja, Pardini and Stoppino \cite{Barja_Pardini_Stoppino2} in higher dimensions. Moreover, surfaces of maximal Albanese dimension attaining the equality in \eqref{eq: Severi} have been explicitly characterized by Barja, Pardini and Stoppino \cite{Barja_Pardini_Stoppino_Severi_Line}, and independently by Lu and Zuo \cite{Lu_Zuo}.

For general irregular varieties, the optimal Severi type inequality is known only for surfaces. Let $S$ be a smooth and irregular surface of general type. Bombieri \cite{Bombieri} showed that the inequality
\begin{equation} \label{eq: Bombieri}
    \vol(S) \ge 2\chi(\omega_S)
\end{equation}
holds. In \cite{Horikawa5}, Horikawa completely described the structure of all irregular surfaces $S$ of general type attaining the equality. More concretely, the Albanese map of $S$ induces a fibration by curves of genus two, the canonical model of $S$ is a flat double cover over an elliptic ruled surface via its relative canonical map with respect to the Albanese fibration, and $\pi_1(S) \simeq \ZZ^2$.

By virtue of the importance of the geographical classification and also parallel to the Noether inequality problem, the following two questions arise naturally:
\begin{itemize}
    \item [(Q1)] What is the analogue of \eqref{eq: Bombieri} for general irregular varieties of general type of dimension $n \ge 3$ (See also \cite[\S1, Question (1)]{Jiang})?

    \item [(Q2)] Once we obtain the optimal $n$-dimensional inequality in (Q1), can we describe the geometry of varieties for which the equality hold?
\end{itemize}  
%In fact, Question (Q1) has been raised already by Z. Jiang as follows:
%\begin{ques*} \cite[Page 2, Question (1)]{Jiang}
    %Does there exist a Severi type inequality for general irregular varieties of dimension $n$?
%\end{ques*}
%Note that this question is open even when $n=3$.

\subsection{Main results}
The first main result of this paper is to establish an optimal three dimensional version of \eqref{eq: Bombieri}, thus answers Question (Q1) in dimension three.

\begin{theorem}[Theorem \ref{thm: Noether-Severi}, \ref{thm: Cartier} and \ref{thm: fudamental group}] \label{thm: main1}
Let $X$ be a smooth and irregular $3$-fold of general type. Then we have the following optimal inequality
\begin{equation} \label{eq: main1 inequality}
	\vol(X) \ge \frac{4}{3} \chi(\omega_{X}).
\end{equation}
If the equality holds, then 
\begin{itemize}
	\item [(1)] $q(X)=1$, $h^2(X, \CO_X)=0$, and the general Albanese fiber of $X$ is a smooth surface $F$ with $\vol(F) = 1$ and $p_g(F)=2$;
	\item [(2)] all minimal models of $X$ are Gorenstein;
	\item [(3)] the topological fundamental group $\pi_1(X) \simeq \ZZ^2$.
\end{itemize}
\end{theorem}

As the title suggests, we call \eqref{eq: main1 inequality} the \emph{Noether-Severi inequality}, because it is an analogue of the Noether inequality \eqref{eq: Noether3} and also of Severi type. Note that by \cite[Remark 3.6]{Hu_Zhang}, for any integer $e > 0$, there exists a smooth and irregular $3$-fold $X$ of general type with $\vol(X) = 4e$ and $\chi(\omega_X) = 3e$.

Previously, \eqref{eq: main1 inequality} was proved by the first named author \cite{Hu} under the extra assumption that $X$ has a Gorenstein minimal model. Unfortunately, the method therein does not work in the general setting. Here in Theorem \ref{thm: main1}, we not only establish \eqref{eq: main1 inequality} in general, but also provide basic properties of the equality case in \eqref{eq: main1 inequality}. An upshot in the characterization is that, \emph{all minimal models of $X$ are Gorenstein}. This is a bit unexpected, at least to us. It suggests that the relation among birational invariants of a $3$-fold may put constraints on singularities on its minimal models, which is not seen in the surface case. It is also a crucial ingredient for us to obtain the explicit structure of $X$ attaining the equality in \eqref{eq: main1 inequality}.

\begin{defi} \label{def: NS line}
    A minimal and irregular $3$-fold $V$ of general type is \emph{on the Noether-Severi line}, if $K_V^3 = \frac{4}{3} \chi(\omega_V)$.
\end{defi}

%As has been done for irregular surfaces $S$ with $\vol(S) = 2\chi(\omega_S)$ in \cite{Horikawa5}, it is natural to ask, \emph{is there a classification of all irregular $3$-folds $X$ on the Noether-Severi line}? This question is also parallel to \cite[Question 1.5]{Chen_Chen_Jiang} for $3$-folds $X$ satisfying $\vol(X) = \frac{4}{3}p_g(X) - \frac{10}{3}$. Unlike surfaces, $X$ may have many minimal models. However, its canonical model is unique. Thus to answer this question, it suffices to characterize the canonical model of $X$.

As the second main result of this paper, we give a complete description of canonical models of all irregular $3$-folds of general type on the Noether-Severi line, thus answer Question (Q2) in dimension three.

\begin{theorem} [Theorem \ref{thm: classification}] \label{thm: main2}
Let $X$ be the canonical model of a minimal and irregular $3$-fold of general type on the Noether-Severi line, with its Albanese fibration $a: X \to B$, where $B$ is a smooth curve of genus one (see Theorem \ref{thm: main1} (1)). Let $X'$ be the blow-up of $X$ along the base-locus section $\Gamma$ of $a$ (see Definition \ref{def: section}). Then the Albanese fibration $a': X' \to B$ of $X'$ is factorized as
$$
a': X' \stackrel{f} \longrightarrow Y \stackrel{q} \longrightarrow S \stackrel{p} \longrightarrow B
$$
with the following properties:
\begin{itemize}
	\item [(1)] $S = \PP_B(a_* \omega_{X}) $ is a $\PP^1$-bundle over $B$ with $p$ the projection.
	
	\item [(2)] $Y = \PP_S (\CO_S \oplus (\CO_S(-2) \otimes \CK_1^{2}))$ is a $\PP^1$-bundle over $S$ with $q$ the projection. Here $\gamma: B \to X$ denotes the section $\Gamma$, and $\CK_1 = p^*(\gamma^*\omega_X)$.
	
	\item [(3)] $f: X' \to Y$ is a flat double cover with the branch locus 
	$$
	D = D_1 + D_2,
	$$ 
	where $D_1 \in |\CO_Y(1)|$, $D_2 \in |\CO_Y(5) \otimes q^* (\CO_S(10) \otimes \CK_1^{-4} \otimes \CK_2^{-2})|$ and $D_1 \cap D_2 = \emptyset$.	Here $\CK_2 = p^*(\det a_* \omega_X)$.	
\end{itemize}

In one word, $X$ is a divisorial contraction of a double cover over a two-tower $\PP^1$-bundle over a smooth curve of genus one.
\end{theorem}

%We would like to add two remarks on Theorem \ref{thm: main2}.
%\begin{itemize}
%\item The section $\Gamma$ can be described very explicitly. Indeed, the general fiber $F$ of $a$ is a canonical $(1, 2)$-surface (Theorem \ref{thm: main1}), and $\Gamma \cap F$ is exactly the unique base point of $|K_F|$. Moreover, $\Gamma$ is contained in the smooth locus of $X$ (Proposition \ref{prop: |K_Xc|}).
%\end{itemize}

%We refer the reader to \S \ref{section: classification} for more details.

In contrast with the rich understanding of the surface classification, to the best of our knowledge, classifications of irregular $3$-folds $X$ of general type with prescribed birational invariants are \emph{quite rare}. The only known results are about those of maximal Albanese dimension with small $\chi(\omega_X)$, in which the generic vanishing theory is essential:
\begin{itemize}
    \item J. Chen, Debarre and Z. Jiang \cite{Chen_Debarre_Jiang} have obtained, up to \'etale covers, an explicit description of $3$-folds $X$ of maximal Albanese dimension and of general type with $\chi(\omega_X) = 0$. 

    \item Based on the work of Barja, Pardini and Stoppino \cite{Barja_Pardini_Stoppino}, Z. Jiang \cite{Jiang} completely described $3$-folds $X$ of maximal Albanese dimension with $\chi(\omega_X) = 1$ and (the smallest volume) $\vol(X) = 12$.
\end{itemize}
Theorem \ref{thm: main2} is the first explicit description so far of an unbounded family of irregular $3$-folds $X$ of general type with $\chi(\omega_X) > 0$, and it is based on a very detailed study of the relative canonical map of $X$ with respect to its Albanese fibration.

%In the first version of this paper, via the study of the canonical map of $X$, we prove Theorem \ref{thm: main1} and a weaker version of Theorem \ref{thm: main2}. Traditionally, the canonical map is a more natural way for such a geographical problem in dimension three (see \cite{Hu,Chen_Chen_Jiang} for example). However, to obtain Theorem \ref{thm: main2} via the canonical map, we do need a ``good" canonical image. At a cost, we can only classify $X$ up to a double cover. The relative canonical map helps remove this ambiguity, because the relative canonical image of $X$ on the Noether-Severi line with respect to its Albanese fibration $a: X \to B$ is automatically smooth. This relative method involves the study of the linear system like $|K_X + a^*T|$ which, at first glance, is not that canonical, as we need to choose a sufficiently ample divisor $T$ on $B$, but we manage to show that the contribution of $T$ disappears when we finally prove Theorem \ref{thm: main1} and \ref{thm: main2}. Moreover, the proof via this method turns out to be more direct and more self-contained, compared with the first version in which we have to invoke some other results in \cite{Tu,Pareschi_Popa_Schnell}.

In a forthcoming paper \cite{Hu_Zhang_Noether}, we will adopt the idea in this paper to study $3$-folds on the Noether line, i.e., those attaining the equality in \eqref{eq: Noether3}.

\subsection{A corollary} Let $X$ be a smooth and irregular variety. Let 
$$
a: X \to A
$$
be the Albanese map of $X$, with $A$ the Albanese variety of $X$. For a divisor $L$ on $X$, its continuous rank $h^0_a(X, L)$ is defined as
$$
h^0_a(X, L): = \min \left\{h^0(X, \CO_X(L) \otimes a^* \alpha) | \alpha \in \Pic^0(A) \right\}.
$$
It has been discovered in the work of Barja, Pardini, Stoppino and Z. Jiang \cite{Barja,Barja_Pardini_Stoppino,Barja_Pardini_Stoppino2,Jiang} that when $h_a^0(X, L)>0$, the slope $\lambda (L):=\frac{\vol(L)}{h^0_a(X, L)}$ is closely related to the geometry of $X$, especially when $L = K_X$. From this perspective, finding the optimal lower or upper bound of $\lambda(K_X)$ becomes an important problem.

When $X$ is of maximal Albanese dimension, by the Severi inequality \eqref{eq: Severi}, we have the optimal lower bound $\lambda(K_X) \ge 2 (\dim X)!$. For general irregular varieties, the corresponding optimal lower bound is known only in dimension two. More precisely, let $S$ be a smooth and irregular surface of general type. Then we have the following optimal lower bound $\lambda(K_S) \ge 2$, i.e., 
\begin{equation} \label{eq: Severi type2}
    \vol(S) \ge 2h^0_a(S, K_S),
\end{equation}
which is an easy consequence of the Noether inequality \eqref{eq: Noether2}. Note that $h^0_a(S, K_S) \ge \chi(\omega_S)$ and that $\chi(\omega_S) = p_g(S) \ge h^0_a(S, K_S)$ when $q(S) = 1$. It follows that $\vol(S) = 2h^0_a(S, K_S)$ if and only if $\vol(S) = 2 \chi(\omega_S)$. Thus by \cite{Horikawa5}, surfaces attaining the equality in \eqref{eq: Severi type2} can be fully characterized. 

Based on Theorem \ref{thm: main1} and \ref{thm: main2}, we establish the optimal lower bound of $\lambda(K_X)$ in dimension three and completely describe the equality case.
\begin{coro}[Corollary \ref{coro: Severi type}] \label{coro: main3} 
Let $X$ be a minimal and irregular $3$-fold of general type. Then we have the following optimal inequality 
\begin{equation} \label{eq: Severi type}
	\vol(X) \ge \frac{4}{3} h^0_a(X, K_X).
\end{equation}
Moreover, the equality holds if and only if $X$ is on the Noether-Severi line.
\end{coro}

Prior to \eqref{eq: Severi type}, the best unconditional result was proved by Barja \cite{Barja} stating that $\vol(X) \ge h^0_a(X, K_X)$ for irregular $3$-folds $X$ of general type. Thus \eqref{eq: Severi type} improves Barja's result. Note that Corollary \ref{coro: main3} offers more. Notably, it shows that Theorem \ref{thm: main2} also gives a complete description of canonical models of those $X$ attaining the equality in \eqref{eq: Severi type}.

\subsection{Outline} The paper is organized as follows. In \S  \ref{section: geometry}, we study the relative canonical map of minimal $3$-folds fibered by $(1, 2)$-surfaces. The key results in this part are Proposition \ref{prop: Noether (1,2)-surface} and \ref{prop: equality} which not only establish some basic numerical inequalities for these $3$-folds but also describe the equality case. In \S \ref{section: Noether-Severi inequality}, based on Proposition \ref{prop: Noether (1,2)-surface}, we prove a part of Theorem \ref{thm: main1}, from which it is sufficient to deduce Corollary \ref{coro: main3}. In \S \ref{section: property}, based on Proposition \ref{prop: equality}, we complete the proof of Theorem \ref{thm: main1}. Moreover, we obtain some very explicit properties about the relative canonical map of minimal $3$-folds on the Noether-Severi line (see Theorem \ref{thm: |K_X|}). Finally, using these properties, we prove Theorem \ref{thm: main2} in \S \ref{section: classification}.

\subsection{Notations and Preliminary} \label{subsection: notation}
In this paper, we adopt the following notation and definitions.

\subsubsection*{Varieties and divisors} Let $V$ be a normal variety of dimension $d$. The \emph{geometric genus} $p_g(V)$ of $V$ is defined as
$$
p_g(V):=h^0(V, K_V).
$$
For a divisor $L$ on $V$, the \emph{volume} $\vol(L)$ of $L$ is defined as
$$
\vol(L) := \limsup\limits_{n \to \infty} \frac{h^0(V, nL)}{n^d/d!},
$$
The volume $\vol(K_V)$ is called the \emph{canonical volume} of $V$, and is denoted by $\vol(V)$. We say that $V$ is \emph{minimal}, if $V$ has at worst $\QQ$-factorial terminal singularity and $K_V$ is nef. If $V$ has at worst canonical singularities and $\vol(V) > 0$, we say that $V$ is of general type. If $V$ is of general type and $K_V$ is ample, we say that $V$ is a canonical model. Note that if $V$ is minimal or a canonical model, then $\vol(V) = K_V^d$.

For a linear system $\Lambda$, we denote by $\movable \Lambda$ and $\baselocus \Lambda$ the movable part and the base locus of $\Lambda$, respectively.

By an \emph{$(a, b)$-surface}, we mean a normal surface $S$ having at worst canonical singularities with $\vol(S) = a$ and $p_g(S) = b$. 

\subsubsection*{Irregular varieties} Let $V$ be a normal variety with at worst rational singularities. We say that $V$ is \emph{irregular}, if $q(V) := h^1(V, \CO_V) > 0$. Note that $V$ has a well-defined \emph{Albanese map}
$$
a: V \to \mathrm{Alb}(V),
$$
where $A:=\mathrm{Alb}(V)$ is an abelian variety referred as the Albanese variety of $V$. Let $V \stackrel{f}\rightarrow W \stackrel{g}\rightarrow A$ be the Stein factorization of $a$. Then $f$ is called the \emph{Albanese fibration} of $V$, and a general fiber of $f$ is called an \emph{Albanese fiber}. If $\phi: V' \to V$ is a resolution of singularities of $V$ and $a': V' \to \mathrm{Alb}(V')$ is the Albanese map of $V'$, then $\mathrm{Alb}(V') \cong \mathrm{Alb}(V)$ and $a' = a \circ \phi$. We refer the reader to \cite[\S 2.4]{Beltrametti_Sommese} for details.

\subsection{Birational modification with respect to linear systems} \label{section: modification}
Let $X$ be a minimal $3$-fold of general type. Let 
$$
\alpha: X_0 \to X
$$ 
be a resolution of singularities of $X$, where $X_0$ is smooth. We may further assume that $\alpha$ is an isomorphism over the smooth locus of $X$.
Let $L$ be a Weil divisor on $X$ with $h^0(X, L) \ge 2$. Since $L$ is $\QQ$-Cartier, we have $h^0(X_0, \rounddown{\alpha^*L}) = h^0(X, L) \ge 2$. Let $|M_0|=\movable|\rounddown{\alpha^*L}|$ be the movable part of $|\rounddown{\alpha^*L}|$. By Hironaka's big theorem, we may resolve the base locus of $|M_0|$ by taking successive blow-ups as follows:
$$
\beta \colon X'=X_{n+1}\stackrel{\pi_n}\rightarrow X_{n}\rightarrow \cdots \rightarrow X_{i+1} \stackrel{\pi_{i}}\rightarrow X_{i}\rightarrow \cdots \rightarrow X_1\stackrel{\pi_0}\rightarrow X_0.
$$
Here each $\pi_i$ is a blow-up along a nonsingular center $W_i$ contained in the base locus of $\mathrm{Mov}|(\pi_0\circ \pi_1 \circ \cdots \circ \pi_{i-1})^*M_0|$. As a result, the linear system $|M|=\mathrm{Mov}|\beta^* M_0|$ is base point free. Set 
$$
\pi: = \alpha \circ \beta: X'\to X.
$$
Such a birational modification will be used throughout this paper.

\subsection*{Acknowledgement}

Y.H. would like to thank Professors Meng Chen and Yongnam Lee for their hospitality and support. Both authors are grateful to Professors Jin-Xing Cai, Jungkai Alfred Chen and Jun Lu for stimulating questions and fruitful discussions. Both authors also would like to thank the anonymous referee sincerely for valuable comments which helped to improve the presentation of the paper dramatically.

Y.H. is supported by a KIAS Individual Grant (MP 062501)
at Korea Institute for Advanced Study and the Shanghai Pujiang Program Grant No. 21PJ1405200. T.Z. is supported by the Science and Technology Commission of Shanghai Municipality (STCSM) Grant No. 18dz2271000 and the National Natural Science Foundation of China (NSFC) General Grant No. 12071139.

\section{Geometry of $3$-folds fibered by $(1, 2)$-surfaces} \label{section: geometry}

Throughout this section, we always assume that $X$ is a minimal $3$-fold of general type with $p_g(X) \ge 4$ and with a fibration 
$$
f: X \to B
$$ 
over a smooth curve $B$ such that the general fiber $F$ of $f$ is a $(1, 2)$-surface.

\subsection{Setting} \label{subsection: setting (1,2)-surface}
By the assumption, $F$ is a minimal $(1, 2)$-surface. Thus $|K_F|$ has a unique base point. Therefore, the horizontal part of the base locus of the relative canonical map of $X$ with respect to $f$ is a natural section $\Gamma$ whose intersection $\Gamma \cap F$ with $F$ is just the base point of $|K_F|$. Since $\Gamma$ is a section of $f$, we have $\Gamma \cong B$. Thus $\Gamma$ is smooth.

Let $T$ be an effective divisor on $B$ with $t = \deg T \ge 0$. Taking a birational modification $\pi: X' \to X$ as in \S \ref{section: modification} with respect to the linear system $|K_X+f^*T|$, we may write
\begin{equation} \label{eq: modification}
\pi^*(K_X+f^*T) = M + Z,
\end{equation}
where $|M| = \movable|\rounddown{\pi^*(K_X+f^*T)} |$  is base point free and $Z$ is an effective $\QQ$-divisor. Note that $X$ has at worst terminal singularities. Thus 
\begin{equation} \label{eq: exceptional divisor}
K_{X'} = \pi^*K_X + E_{\pi},
\end{equation}
where $E_{\pi}$ is an effective $\pi$-exceptional $\QQ$-divisor.

Let $\phi_{K_X + f^*T}: X \dashrightarrow \PP^{h^0(X, K_X+f^*T)-1}$ be the rational map of $X$ induced by $|K_X + f^*T|$, with the image $\Sigma$. Then we have the following commutative diagram
$$
\xymatrix{
& & X' \ar[d]_{\pi} \ar[lld]_{f'} \ar[rr]^{\psi} \ar[drr]^{\phi_{M}} & &  \Sigma' \ar[d]^{\tau}  \\
B & & X  \ar@{-->}[rr]_{\phi_{K_X + f^*T}} \ar[ll]^f  & & \Sigma         
}
$$
where $\phi_M: X' \to \Sigma$ is the morphism induced by $|M|$, $X' \stackrel {\psi}\rightarrow \Sigma' \stackrel{\tau} \rightarrow \Sigma$ is the Stein factorization of $\phi_M$, and $f' = f \circ \pi$ is the fibration of $X'$. Denote by $F'$ a general fiber of $f'$. Since $p_g (F)=2$, we have $\dim \phi_M (F') \le 1$. Thus $\dim \Sigma \le 2$. 

\begin{remark} \label{rmk: T}
We would like to point out that the varieties as well as the maps in above diagram may vary when the divisor $T$ varies. However, for simplicity, we will use the same set of notation as above as a universal one. For the convenience of the reader, each time we will specify clearly which $T$ we use in the corresponding setting. In particular, by $T$ being sufficiently ample, we mean that $h^0(X, K_X+f^*T)\ge g(B)+2$ and $\phi_{K_X + f^*T}$ is the relative canonical map $X \dashrightarrow \PP_B(f_* \omega_{X/B})$ of $X$ with respect to $f$.
\end{remark}

\subsection{The case when $\dim \Sigma = 2$} \label{subsection: Noether (1,2)-surface}
In this subsection, we always assume that $\dim \Sigma =2$, and we have no a priori restriction on $T$. In fact, when $T$ is sufficiently ample, we always have $\dim \Sigma = 2$. 

Let $C$ be a general fiber of $\psi$. Then we have
$$
M^2 \equiv dC,
$$
where $d := (\deg \tau) \cdot (\deg \Sigma)$. Since $\Sigma \subset \PP^{h^0(X, K_X+f^*T) - 1}$ is a non-degenerate surface, we have $\deg \Sigma \ge h^0(X, K_X + f^*T) - 2$.

\begin{lemma} \label{lem: g2}
We have $|M||_{F'} = \movable |(\pi|_{F'})^*K_F|$, where $\pi|_{F'}: F' \to F$ coincides with the blow-up of the unique base point of $|K_F|$. In particular, $g(C)=2$ and $f'$ factors through $\psi$ birationally (i.e., there is a rational map $p_{\Sigma'}: \Sigma'\dashrightarrow B$ such that $f'=p_{\Sigma'}\circ\psi$).
\end{lemma}

\begin{proof}
Since $p_g(F)=2$ and $\phi_{M}(F')$ is a curve, the restriction map
$$
H^0(X, K_X+f^*T)\to H^0(F, K_F)
$$
is surjective. Thus the horizontal part of $\baselocus|K_X+f^*T|$ with respect to $f$ is just the section $\Gamma$. From the construction of $\pi$, we see that $\pi|_{F'}: F' \to F$ is just the blow-up of the unique base point of $|K_F|$. It follows that $|M||_{F'} = \movable |(\pi|_{F'})^*K_F|$. Note that $\movable |(\pi|_{F'})^*K_F|$ is a rational pencil of curves of genus two. We deduce that $g(C)=2$. Moreover, $(F' \cdot C)=0$. Thus $f'$ factors birationally through $\psi$.
\end{proof}

The following estimate will be used when we treat the canonical map or the relative canonical map of $X$.

\begin{lemma} \label{lem: d}
The following statements hold:
\begin{itemize}
	\item [(1)] If $T=0$, then we have
	$$
	d \ge \min \{2p_g(X)-4,\ p_g(X) + g(B) -2 \} \ge p_g(X) - 2.
	$$
	\item [(2)] If $g(B)>0$ and $T$ is sufficiently ample as in Remark \ref{rmk: T}, then we have 
	$$
	\deg \Sigma \ge \max \{h^0(X, K_X + f^*T), h^0(X, K_X+f^*T) + g(B) - 2\}.
	$$
\end{itemize}

\end{lemma}

\begin{proof}
For (1), the second inequality is obvious. Thus we only need to prove the first one. Take a general member $S \in |M|$. By Bertini's theorem, $S$ is smooth. Moreover, by Lemma \ref{lem: g2}, $S \cap F' = C$. Now a general fiber of the induced morphism $f'|_S: S \to B$ is $S \cap F' = C$, and it is just a general fiber of $\psi|_S$. We conclude that $f'|_S$ and $\psi|_S$ are identical to each other.

Since $M|_S \equiv dC$, we may write $M|_S = (\psi|_S)^*D$, where $D$ is an effective divisor on $B$ of degree $d$. Note that $h^0(B, D) = h^0(S, M|_S) \ge p_g(X) - 1$. If $h^1(B, D)>0$, by Clifford's inequality, 
$$
\deg D \ge 2h^0(B, D)-2 \ge 2p_g(X)-4.
$$
If $h^1(B, D)=0$, by the Riemann-Roch theorem, we have
$$
\deg D = h^0(B, D)+ g(B)-1 \ge p_g(X)+g(B)-2.
$$
This proves (1).

For (2), by the assumption on $T$, the map $\phi_{K_X+f^*T}$ is the relative canonical map $X \dashrightarrow \PP_B(f_* \omega_{X/B})$ of $X$ with respect to $f$. In particular, $\Sigma = \mathbb{P}_B(f_*\omega_{X/B}) \subset \mathbb{P}^{h^0(X, K_X+f^*T)-1}$ is a smooth $\PP^1$-bundle over $B$. Since $g(B) > 0$, $\Sigma$ is irregular. Thus by \cite[\S 10 and Theorem 8]{Nagata}, we deduce that 
$$
\deg \Sigma \ge h^0(X, K_X + f^*T).
$$

For the rest part of (2), let $H$ be the restriction on $\Sigma$ of a general hyperplane of $\mathbb{P}^{h^0(X, K_X+f^*T) - 1}$. By Bertini's theorem, $H$ is smooth. Moreover, $\deg \Sigma = \deg \CO_H(H)$. Since $H$ has a natural cover to $B$ induced by the projection, we have $g(H) \ge g(B)$. If $h^1(H, \CO_H(H)) > 0$, by Clifford's inequality,
$$
\deg \Sigma \ge 2 h^0(H, \CO_H(H)) - 2 \ge 2 h^0(X, K_X+f^*T) - 4.
$$
If $h^1(H, \CO_H(H)) = 0$, by the Riemann-Roch theorem, we have
$$
\deg \Sigma = h^0(H, \CO_H(H)) + g(H) - 1 \ge h^0(X, K_X+f^*T) + g(B) - 2.
$$
By Remark \ref{rmk: T}, we have $h^0(X, K_X+f^*T) \ge g(B) + 2$. It follows that
$$
\deg \Sigma \ge h^0(X, K_X+f^*T) + g(B) - 2.
$$
Thus the proof is completed.
\end{proof}

\begin{lemma} \label{lem: E0} 
There exists a unique $\pi$-exceptional prime divisor $E_0$ on $X'$ such that 
\begin{itemize}
	\item [(1)] $\coeff_{E_0} (Z) = \coeff_{E_0}(E_{\pi}) = 1$;
	\item [(2)] $\pi(E_0)=\Gamma$, $\phi_M (E_0) = \Sigma$;
	\item [(3)] $(E_0 \cdot C) = (Z \cdot C) = (E_{\pi} \cdot C) = ((\pi^*K_X) \cdot C) = 1$.
\end{itemize}
\end{lemma}

\begin{proof}
By Lemma \ref{lem: g2}, we may assume that $C$ is a general member of $\movable |(\pi|_{F'})^*K_F|$. Then we have 
$$
\left((\pi^*K_X) \cdot C\right) = \left((\pi^*K_X)|_{F'} \cdot C\right) = \left(\left((\pi|_{F'})^*K_F\right) \cdot C \right) =1.
$$
Since $(M \cdot C)=0$, it follows that
$$
(Z|_{F'} \cdot C) = (Z \cdot C) = \left((\pi^*(K_X+f^*T))  \cdot C\right)=\left((\pi^*K_X)  \cdot C \right) =1.
$$

By Lemma \ref{lem: g2}, $Z|_{F'}$ is a $(-1)$-curve on $F'$. Thus there exists a unique prime divisor $E_0 \subseteq Z$ with $\coeff_{E_0} (Z)=1$ such that $E_0|_{F'} = Z|_{F'}$. In particular, $\pi(E_0) = \Gamma$. Moreover, 
$$
(E_0 \cdot C) = (E_0|_{F'} \cdot C) = (Z|_{F'} \cdot C) = 1.
$$
Thus $\phi_{M}(E_0) = \Sigma$.

By the adjunction formula,
$$
K_{F'} = K_{X'}|_{F'} = (\pi^*K_X)|_{F'} + E_{\pi}|_{F'} = (\pi|_{F'})^*K_F + E_{\pi}|_{F'}.
$$
By Lemma \ref{lem: g2}, we deduce that $E_{\pi}|_{F'}$ is just the $(-1)$-curve $Z|_{F'}$. Using the same argument as for $Z$, we deduce that $\coeff_{E_0} (E_{\pi}) = 1$ and $(E_{\pi} \cdot C)=1$. The proof is completed.
\end{proof}

\begin{prop} \label{prop: Noether (1,2)-surface}
Suppose either $p_g(X) \ge 4$ or $T$ is sufficiently ample as in Remark \ref{rmk: T}. Then the following inequalities hold:
\begin{itemize}
	\item [(1)] $\displaystyle{(K_X\cdot\Gamma)  \ge \frac{1}{3} d - \frac{2}{3} t + \frac{2}{3}(g(B)-1)}$;
	\item [(2)] $\displaystyle{K_X^3 \ge \frac{4}{3} d-\frac{8}{3}t+\frac{2}{3}(g(B)-1)}$.
\end{itemize}
\end{prop}

\begin{proof}
Let $E_0$ be the unique $\pi$-exceptional prime divisor as in Lemma \ref{lem: E0}. Similarly as in the proof of Lemma \ref{lem: d}, take a general member $S \in |M|$. By Bertini's theorem and Lemma \ref{lem: E0} (2), $S$ is smooth and $E_0|_S$ is irreducible. Then we obtain the fibration $f'|_S = \psi|_S: S \to B$.

Denote $\Gamma_S = E_0|_S$. Then $\Gamma_S$ is a section of $\psi|_S$. By Lemma \ref{lem: E0} (1), we may write 
\begin{equation} \label{eq: EVZV}
	E_{\pi}|_S = \Gamma_S + E_V, \quad Z|_S = \Gamma_S + Z_V.
\end{equation}
Here $E_V$ and $Z_V$ are effective $\QQ$-divisors on $S$. By Lemma \ref{lem: E0} (3), 
$$
(E_V \cdot C) = \left((E_{\pi} - E_0) \cdot C\right) = 0, \quad (Z_V \cdot C) = \left((Z-E_0) \cdot C\right)=0.
$$ 
We deduce that both $E_V$ and $Z_V$ are vertical with respect to $\psi|_S$.

Note that $(f'^*T)|_S \equiv tF'|_S \equiv tC$. By the adjunction formula, \eqref{eq: modification}, \eqref{eq: exceptional divisor} and \eqref{eq: EVZV},
\begin{equation}\label{eq: K_S}
	K_S = (K_{X'} + S)|_S \equiv (2d - t)C + 2 \Gamma_S + E_V + Z_V.
\end{equation}
Moreover, since $p_g(X)>0$, we have $M \ge f'^*T$. Thus $dC\equiv M|_S\ge (f'^*T)|_S$.  We deduce that $d \ge t$. Thus 
$$
2d - t \ge d \ge h^0(S, M|_S) - 1 \ge h^0(X, K_X + f^*T) - 2.
$$
Therefore, if $p_g(X) \ge 4$ or $T$ is sufficiently ample, we always have $2d-t \ge 2$.

Denote by $\sigma: S \to S_0$ the contraction onto its minimal model $S_0$. By \cite[Corollary 2.3 for $\lambda = 1$ and $D = K_X+f^*T$]{Chen_Chen_Jiang},\footnote{Here $D$ is semi-ample by \cite[Theorem 3.3]{Kollar_Mori}.} we have 
\begin{equation} \label{eq: restriction comparison}
	\left. \left(\pi^* \left(K_X+\frac{1}{2} f^*T \right)\right) \right|_S \sim_{\QQ} \frac{1}{2} \sigma^* K_{S_0} + H,
\end{equation}
where $H$ is an effective $\QQ$-divisor. Therefore, by Lemma \ref{lem: E0} (3), $((\sigma^*K_{S_0}) \cdot C) \le ((\pi^*(2K_X + f^*T)) \cdot C) = 2$. Let $C_0 = {\sigma}_*C$ and $\Gamma_{S_0} = {\sigma}_*\Gamma_S$. Then we have $(K_{S_0} \cdot C_0) \le 2$. On the other hand, by \eqref{eq: K_S},
$$
K_{S_0} \equiv (2d - t)C_0 + 2\Gamma_{S_0} + {\sigma}_*(E_V + Z_V).	
$$
We deduce that $(K_{S_0} \cdot C_0) \ge (2d -t) C_0^2 \ge 2 C_0^2$. By parity, it follows that $(K_{S_0} \cdot C_0) = 2$ and $C_0^2 = 0$. In particular, the fibration $\psi|_S$ descends to a fibration $S_0 \to B$ whose general fiber is $C_0$ with $g(C_0)=2$. Moreover, 
$$
(H \cdot C) = \left( \left(\pi^*\left(K_X + \frac{1}{2} f^*T\right) \right) \cdot C \right)-\frac{1}{2}(K_{S_0} \cdot C_0) = 0.
$$ 
This implies that $H$ is vertical with respect to $\psi|_S$. 

Since $\Gamma_{S_0}$ is a section of the fibration $S_0 \to B$, $\Gamma_{S_0} \simeq B$. Thus $g(\Gamma_{S_0}) = g(B)$. By the adjunction formula on $S_0$, we have
\begin{align}
	2g(B)-2 & =(K_{S_0} \cdot \Gamma_{S_0}) + \Gamma_{S_0}^2 \nonumber \\
	& = (K_{S_0} \cdot \Gamma_{S_0}) + \frac{1}{2} \left( \left(K_{S_0} - (2d - t) C_0 - {\sigma}_*(E_V + Z_V) \right) \cdot \Gamma_{S_0}\right) \nonumber \\
	& =\frac{3}{2} (K_{S_0} \cdot \Gamma_{S_0}) - \left (d -\frac{1}{2}t \right) - \frac{1}{2} \left({\sigma}_*(E_V+Z_V) \cdot \Gamma_{S_0} \right) \label{eq: remark 1} \\
	&\le \frac{3}{2} (K_{S_0} \cdot \Gamma_{S_0}) - \left (d -\frac{1}{2} t \right), \nonumber
\end{align}
where the last inequality follows from the fact that ${\sigma}_*(E_V + Z_V)$ is vertical with respect to the fibration $S_0 \to B$. Together with \eqref{eq: restriction comparison} and the fact that $(H \cdot \Gamma_S) \ge 0$, we deduce that
\begin{align*}
	(K_X \cdot \Gamma)  = \left((\pi^*K_X)|_S \cdot \Gamma_S \right) & \ge \frac{1}{2} \left((\sigma^* K_{S_0}) \cdot \Gamma_S \right)-\frac{1}{2} t \\
	& \ge \frac{1}{3} d-\frac{2}{3} t+\frac{2}{3}(g(B)-1).
\end{align*}
This proves (1). 

For (2), note that
\begin{align} 
	K_X^3 & = (K_X \cdot (K_X+f^*T)^2) - 2tK_F^2 \nonumber \\ 
	& \ge ((\pi^* K_X) \cdot M^2) + \left( (\pi^*K_X)|_S \cdot Z|_S \right)-2t. \label{eq: remark 2}
\end{align}
By Lemma \ref{lem: E0} (3), 
$$
\left((\pi^*K_X) \cdot M^2 \right) = d \left((\pi^*K_X) \cdot C \right) = d.
$$
By (1), we have
$$
\left((\pi^*K_X)|_S \cdot Z|_S \right) \ge \left((\pi^*K_X)|_S \cdot \Gamma_S \right) = (K_X \cdot \Gamma) \ge \frac{1}{3} d - \frac{2}{3} t + \frac{2}{3}(g(B)-1).
$$
Combine the above inequalities  together. We deduce that 
$$
K_X^3 \ge  \frac{4}{3}d - \frac{8}{3}t + \frac{2}{3}(g(B)-1).
$$ 
This proves (2).
\end{proof}

The following is a crucial proposition that we will frequently use in our later argument.

\begin{prop} \label{prop: equality}
Keep the same notation as in the proof of Proposition \ref{prop: Noether (1,2)-surface}. Suppose the equality in Proposition \ref{prop: Noether (1,2)-surface} (2) holds. Then we have
\begin{itemize}
	\item [(1)] $\sigma_*(E_V + Z_V) = 0$ and $K_{S_0} \equiv (2d - t)C_0 + 2 \Gamma_{S_0}$;
	%\item [(2)] $\Gamma_{S_0}^2 = -\frac{1}{3}(2d - t) + \frac{2}{3}(g(B)-1)$;
	\item [(2)] $K_{S_0}^2 = \frac{8}{3} (2d - t + g(B) - 1)$;
	\item [(3)] $(\pi^*(2K_X+f^*T))|_S \sim_{\QQ} \sigma^*K_{S_0}$.
\end{itemize}
\end{prop}

\begin{proof} 	
Since the equality in Proposition \ref{prop: Noether (1,2)-surface} (2) holds, all inequalities in the above proof become equalities.  On the one hand, \eqref{eq: remark 1} becomes an equality, and we have $({\sigma}_*(E_V + Z_V) \cdot \Gamma_{S_0}) = 0$. Thus $({\sigma}_*(E_V + Z_V) \cdot (2\Gamma_{S_0} + (2d - t) C_0))=0$. Note that by \cite[Lemma 1]{Bombieri}, $K_{S_0}$ is $2$-connected. We deduce that ${\sigma}_*(E_V + Z_V) = 0$. It follows that
$$
K_{S_0} \equiv (2d - t)C_0 + 2 \Gamma_{S_0}.
$$
Thus (1) is proved. Also by \eqref{eq: remark 1}, we have 
$$
\Gamma_{S_0}^2 = -(K_{S_0} \cdot \Gamma_{S_0}) + 2g(B) - 2 = -\frac{1}{3}(2d - t) + \frac{2}{3}(g(B)-1).
$$ 
Together with (1), it follows that 
$$
K_{S_0}^2 = \left((2d -t) C_0 + 2\Gamma_{S_0}\right)^2 = \frac{8}{3} (2d - t + g(B) - 1),
$$
and (2) is proved. Now \eqref{eq: remark 2} also becomes an equality. This implies that
\begin{align*}
	\left( \left(\pi^*(2K_X + f^*T) \right)|_S \right)^2 
	& = 4 \left( (\pi^*K_X)|_S \cdot \left(\pi^*(K_X + f^*T) \right)|_S \right) \\
	& = 4 \left((\pi^*K_X) \cdot M^2 \right) + 4((\pi^*K_X)|_S \cdot Z|_S) \\ 
	& = \frac{8}{3} (2d - t + g(B) - 1).
\end{align*}
By \eqref{eq: restriction comparison}, we have
\begin{align*}
		\left( \left(\pi^*(2K_X + f^*T) \right)|_S \right)^2 \ge   \left( \left(\pi^*(2K_X + f^*T) \right)|_S  \cdot \sigma^*K_{S_0}\right) \ge K_{S_0}^2.
\end{align*}
Combine the above equalities and inequalities with \eqref{eq: restriction comparison} and apply the Hodge index theorem. We deduce that
$$
(\pi^*(2K_X+f^*T))|_S \equiv \sigma^*K_{S_0},
$$
i.e., $H \equiv 0$. Since $H$ is effective, we conclude that $H=0$ and
$$
(\pi^*(2K_X+f^*T))|_S \sim_{\QQ} \sigma^*K_{S_0}.
$$
This proves (3).
\end{proof}

\subsection{A particular case when $T=0$, $g(B)>0$ and $\dim \Sigma = 2$} \label{subsection: d=pg-1}

In this subsection, we still assume that $\dim \Sigma =2$, but we further assume that $T = 0$, $g(B)>0$ and $p_g(X) \ge 4$. That is, we consider the canonical map of $X$.  By Lemma \ref{lem: d} (1), we have 
$$
d = (\deg \tau) \cdot (\deg \Sigma) \ge p_g(X) - 1.
$$ 
In the following, we focus on the specific case when 
$$
d= p_g(X) - 1.
$$
In this case, since $\deg \Sigma \ge p_g(X) - 2$ and $p_g(X) \ge 4$, we deduce that $\deg \Sigma = p_g(X) - 1$ and $\deg \tau = 1$. Since $f'$ factors birationally through $\psi$ by Lemma \ref{lem: g2}, we further deduce that $\Sigma$ is birationally fibered over $B$. Note that $g(B) > 0$. By \cite[Theorem 8]{Nagata}, $g(B) = 1$ and $\Sigma$ is a cone over $B$. In particular, $\Sigma$ is normal. Thus $\Sigma' = \Sigma$. Moreover, $\phi_M(F')$ is a line on $\Sigma$.

Let $\Sigma_0$ be the blow-up of $\Sigma$ at the unique cone singularity $v$. Then we obtain a ruled surface $p: \Sigma_0 \to B$ over $B$. Let $s$ be the exceptional curve on $\Sigma_0$, and let $l$ be a ruling on $\Sigma_0$. Then the morphism $\Sigma_0 \to \Sigma \subset \PP^{p_g(X) -1}$ is just induced by the base-point-free linear system $|s + p^*e_B|$ on $\Sigma_0$, where $e_B$ is a divisor on $B$ of degree $d=p_g(X) - 1$. Since $s$ is contracted via this morphism, we deduce that $s^2 = -(p_g(X)-1) \le -3$. By \cite[V, Theorem 2.15]{Hartshorne}, $p_* \CO_{\Sigma_0}(s)$ is decomposable. Thus 
$$
p_* \CO_{\Sigma_0}(s) = \CO_B \oplus \CO_B(-e_B).
$$
Therefore, we have the following commutative diagram
$$
\xymatrix{
X' \ar@{-->}[rr]^{\psi_0} \ar[rrd]_{\psi = \phi_M} & & \Sigma_0 \ar[rr]^{p} \ar[d] & & B \\
& & \Sigma \ar@{-->}[rru] & &   
}
$$
where the rational map $\psi_0$ is induced by the blow-up of $\Sigma$.

\begin{lemma} \label{lem: psi_0}
The rational map $\psi_0$ is a morphism. Moreover, $p \circ \psi_0 = f'$ and 
\begin{equation} \label{eq: M}
	|M| = \psi_0^*|s + p^*e_B|.
\end{equation}
\end{lemma}

\begin{proof}
Take a birational modification $\nu: X'' \to X'$ such that:
\begin{enumerate}
	\item $X''$ is smooth and $\nu$ is an isomorphism over $X''-\psi^{-1}(v)$;
	\item the rational map $\mu = \psi_0 \circ \nu: X''\dashrightarrow \Sigma_0$ is a morphism.
\end{enumerate}
Let $f'': X'' \to B$ be the Albanese fibration of $X''$. Then $f'' = f' \circ \nu$. By Lemma \ref{lem: g2}, $f'$ birationally factors through $\psi$. It follows that $f'' = p \circ \mu$.

Suppose $\psi_0$ is not a morphism. By Zariski's main theorem, there is an integral curve $A \subset X''$ such that $\nu(A) \in \psi^{-1}(v)$ is a point and $\mu(A) \subset \Sigma_0$ is a curve. Since $\psi(\nu(A)) = v$, we deduce that $\mu(A) = s$. Thus $f''(A) = p(s) = B$. On the other hand, $f''(A) = f'(\nu(A))$ is a point. This is a contradiction. Thus $\psi_0$ is a morphism. 

Now $\psi_0^*l = F'$. Thus it follows that $p \circ \psi_0 = f'$, and $\eqref{eq: M}$ simply follows from the above commutative diagram.
\end{proof}

In the following, we adopt the idea in \cite{Chen_Chen_Jiang} to prove the ``weak pseudo-effectivity'' of $3\pi^*K_X- d F'$ (Lemma \ref{lem: weak pseudo-effective}). Recall that by Lemma \ref{lem: E0},  there is a unique $\pi$-exceptional prime divisor $E_0$ satisfying the condition therein. 

\begin{lemma}\label{lem: D0}
There exists a unique prime divisor $D_0$  such that
\begin{itemize}
	\item [(1)] $\coeff_{D_0} (\psi_0^*s) = 1$;
	\item [(2)] $(D_0 \cdot E_0 \cdot F')=1$ and $((\pi^*K_X) \cdot D_0 \cdot F')=1$.
\end{itemize} 
\end{lemma}

\begin{proof}
By the abuse of notation, we still denote by $C$ the general fiber of $\psi_0$. By Lemma \ref{lem: psi_0}, we have 
$$
M|_{F'}=(\psi_0^*(s+p^*e_B))|_{F'}\equiv (\psi_0^*s)|_{F'}.
$$
By Lemma \ref{lem: g2}, $M|_{F'}\equiv C$. Thus we have $(\psi_0^*s)|_{F'}  \equiv C$. By Lemma \ref{lem: E0} (3), $((\psi_0^*s) \cdot E_0 \cdot F')= (E_0 \cdot C) =1$. Since $\psi_0^*s$ is Cartier and $E_0 \nsubseteq \mathrm{Supp}(\psi_0^*s)$, for any prime divisor $D$ with $\coeff_{D} (\psi_0^*s) > 0$, we have $\coeff_{D}(\psi_0^*s) \ge 1$ and $(D \cdot E_0 \cdot F')$ is a non-negative integer. Thus there exists a unique prime divisor $D_0$ with $\coeff_{D_0} (\psi_0^*s)  = 1$ such that 
$$
(D_0 \cdot E_0 \cdot F')=1.
$$ 
Note that we have $(\pi^*K_X)|_{F'}=(\pi|_{F'})^*K_F\equiv C+E_0|_{F'}$ by Lemma \ref{lem: g2}. Therefore, 
$$
\left((\pi^*K_X) \cdot D_0 \cdot F'\right) = (E_0\cdot D_0\cdot F') = 1.
$$
The proof is completed.
\end{proof}

\begin{lemma}\label{lem: construction of effective divisor (1,2) surface}
Let $A$ be an ample Cartier divisor on $\Sigma_0$. For any integer $m>0$, there exists an integer $c>0$ and an effective divisor $H_m \sim cm(K_{X'/\Sigma_0}+E_0)+ c \psi_0^*A$ such that $E_0 \nsubseteq \mathrm{Supp}(H_m)$, where $K_{X'/\Sigma_0}=K_{X'}-\psi_0^*K_{\Sigma_0}$ and $E_0$ is the $\pi$-exceptional divisor as in Lemma \ref{lem: E0}.
\end{lemma}

\begin{proof}
The proof of \cite[Claim 4.9]{Chen_Chen_Jiang} works verbatim in our setting, and we only need to replace $W$, $g$, $\mathbb{F}_a$ and $E_0$ therein by $X'$, $\psi_0$, $\Sigma_0$ and $E_0$ in our context. We leave the detailed proof to the interested reader.
\end{proof}

\begin{lemma}\label{lem: weak pseudo-effective}
For any nef $\QQ$-divisor $L$ on $X'$, we have
$$
\left((3\pi^*K_X-dF') \cdot D_0\cdot L\right) \ge 0,
$$
where $D_0$ is the divisor as in Lemma \ref{lem: D0}.
\end{lemma}

\begin{proof}
The proof is just a slight modification of that of \cite[Claim 4.10]{Chen_Chen_Jiang}. For the convenience of the reader, we give a detailed proof here.

By \eqref{eq: modification}, \eqref{eq: exceptional divisor} and Lemma \ref{lem: psi_0}, we have
\begin{align}
	K_{X'/\Sigma_0} + E_0 & = (\pi^*K_X + E_\pi) + \psi_0^*(2 s + p^*e_B) + E_0 \nonumber \\ 
	& =  \pi^*K_X + 2M - \psi_0^*(p^*e_B) + E_\pi + E_0  \label{eq: K+E_0} \\ 
	& = 3\pi^*K_X - \psi_0^*(p^*e_B) + E_\pi + E_0 - 2Z. \nonumber 
\end{align}
Write $E_\pi+E_0-2Z=N_{+}-N_{-}$, where $N_{+}$ and $N_{-}$ are both effective $\mathbb{Q}$-divisors with no common components. By Lemma \ref{lem: E0} (1), we deduce that $E_0 \nsubseteq \supp(N_{+})$ and $E_0 \nsubseteq \supp(N_{-})$.

Choose an ample divisor $A=t_1 p^*e_B + t_2 s$ on $\Sigma_0$, where $t_1>t_2$ are two positive integers such that $t_2 K_X$ is Cartier. Let $m$ be a positive integer such that $mK_X$ is Cartier. By Lemma \ref{lem: construction of effective divisor (1,2) surface}, there exists an integer $c>0$ and an effective divisor $H_m\sim cm (K_{X'/\Sigma_0}+E_0) +c \psi_0^*A$ such that $E_0\nsubseteq \mathrm{Supp}(H_m)$. Thus we have 
\begin{align*}
	& \hspace{16 pt} H_m+cmN_{-}+ct_2Z \\
	& \sim cm (K_{X'/\Sigma_0}+E_0) + c\psi_0^*A + cmN_{-} + ct_2Z\\
	& = cm\left(3\pi^*K_X-\psi_0^*(p^*e_B) + E_\pi + E_0 -2Z + N_{-}\right)+ c\psi_0^*A +ct_2Z \hfill \\ 
	& = cm\left(3\pi^*K_X-\psi_0^*(p^*e_B)\right)+c\left(t_1 \psi_0^*(p^*e_B)+t_2(\psi_0^*s + Z)\right)+cmN_{+}\\
	& = cm\left(3\pi^*K_X-\psi_0^*(p^*e_B)\right)+c(t_1-t_2)\psi_0^*(p^*e_B) + ct_2(M+Z) + cmN_{+} \\ 
	& = cm(3\pi^*K_X-\psi_0^*(p^*e_B))+c\left((t_1-t_2)\psi_0^*(p^*e_B) + t_2 \pi^*K_X\right) + cmN_{+}.
\end{align*}
Here the first equality is by \eqref{eq: K+E_0}, and the last two equalities are by \eqref{eq: M} and \eqref{eq: modification}, respectively. By Lemma \ref{lem: psi_0}, $p \circ \psi_0 = f' = f \circ \pi$. Thus $\psi_0^*(p^*e_B) = \pi^*(f^*e_B)$. This implies
\begin{align*}
	& \hspace{16 pt} H_m+cmN_{-}+ct_2Z - cmN_+ \\
	& \sim cm \pi^*(3K_X - f^*e_B) + c \pi^*\left(t_2K_X + (t_1-t_2)f^*e_B\right).
\end{align*}  
Note that $N_+$ is $\pi$-exceptional. We deduce that $cmN_{+}$ is contained in the fixed part of $|H_m+cmN_{-}+ct_2Z|$. In particular,  
$H_m+cmN_{-}+ct_2Z-cmN_{+}$ is effective.

Let $G_m = \frac{1}{cm}(H_m+cmN_{-}+ct_2Z-cmN_{+})$. Since $E_0 \nsubseteq \mathrm{Supp}(H_m) \cup \mathrm{Supp}(N_{+}) \cup \mathrm{Supp}(N_{-})$, by Lemma \ref{lem: E0} (1), $\coeff_{E_0} (G_m) = \frac{t_2}{m} \coeff_{E_0} (Z) = \frac{t_2}{m}$. By Lemma \ref{lem: D0} (2), 
$$
\left( \left(G_m-\frac{t_2}{m}E_0 \right)\cdot E_0 \cdot F'\right) \ge \mu_m (D_0 \cdot E_0\cdot F') = \mu_m,
$$
where $\mu_m = \coeff_{D_0}(G_m)$. Since $E_0$ is $\pi$-exceptional and both $G_m$ and $F'$ are $\pi$-trivial, $(G_m \cdot E_0 \cdot F')=0$. It follows that
$$
-\frac{t_2}{m}(E_0^2 \cdot F')\ge \mu_m \ge 0.
$$
In particular, $\lim\limits_{m \to \infty} \mu_m=0$. Thus for any nef $\QQ$-divisor $L$ on $X'$, we have
$$
\lim\limits_{m \to \infty} (G_m\cdot D_0\cdot L) = \lim\limits_{m \to \infty} \left((G_m-\mu_m D_0) \cdot D_0\cdot L\right) \ge 0.
$$
By the definition of $G_m$, the above inequality just implies that
\begin{align*}
	\left((3\pi^*K_X - dF') \cdot D_0 \cdot L\right) = \lim\limits_{m \to \infty} (G_m\cdot D_0\cdot L) \ge 0.
\end{align*}	
The proof is completed.
\end{proof}

\begin{prop} \label{prop: shaper Noether (1,2)-surface}
Suppose $g(B) > 0$ and $p_g(X) \ge 4$. If $T = 0$, $\deg \Sigma = 2$ and $d = p_g(X) - 1$, then we have
$$
K_X^3 > \frac{4}{3}d = \frac{4}{3}(p_g(X)-1).
$$
\end{prop}

\begin{proof}
By \eqref{eq: modification} and \eqref{eq: M}, $\pi^*K_X \equiv dF' + \psi_0^*s + Z$. Thus
$$
K_X^3 \ge d \left((\pi^*K_X)^2 \cdot F'\right) + \left((\pi^*K_X)^2 \cdot (\psi_0^*s)\right) \ge d + \left((\pi^*K_X)^2 \cdot D_0\right),
$$
where $D_0$ is the unique divisor as in Lemma \ref{lem: D0}. Note that we have $g(B)=1$, we deduce that $K_{X/B} = K_X$. Thus $K_X - \epsilon F$ is nef for some $\epsilon > 0$ by \cite[Theorem 1.4 and Lemma 1.6]{Ohno}. By Lemma \ref{lem: D0} (2) and Lemma \ref{lem: weak pseudo-effective}, 
\begin{align*}
	0 & \le \left((3\pi^*K_X - dF')\cdot D_0\cdot (\pi^*K_X-\epsilon F')\right) \\
	& = 3 \left((\pi^*K_X)^2 \cdot D_0\right) -(3\epsilon + d) \left((\pi^*K_X) \cdot D_0 \cdot F'\right) \\
	& = 3\left((\pi^*K_X)^2 \cdot D_0\right) - (3\epsilon + d).
\end{align*}  
That is, $((\pi^*K_X)^2 \cdot D_0) \ge \frac{d}{3} + \epsilon$. It follows that
$$
K_X^3 \ge d + \frac{d}{3} + \epsilon > \frac{4}{3}d = \frac{4}{3}(p_g(X)-1).
$$
The proof is completed.
\end{proof}

\begin{remark} \label{rmk: K-cF nef}
	Keep the same assumption as in Proposition \ref{prop: shaper Noether (1,2)-surface}. If we further assume that $K_X-cF$ is nef for some $c > 0$, then the same proof leads to the inequality that
	$$
	K_X^3 \ge \frac{4}{3}(p_g(X)-1) + c.
	$$
\end{remark}

\subsection{The case when $T = 0$ and $\dim \Sigma = 1$}  We have the following result.

\begin{prop} \label{prop: canonical pencil}
Suppose that $T = 0$ and $\dim \Sigma = 1$. Then $\psi = f'$. Moreover, if $X$ is irregular, then $q(X) = 1$ and $h^2(X, \CO_X) = 0$.
\end{prop}

\begin{proof}
By the assumption, $p_g(X) \ge 4$. Note that a general fiber $F'$ cannot be contained in the base locus $Z$ of $|\rounddown{\pi^*K_X + f^*T}|$. Thus 
$$
h^0(X', M - F') = h^0(X', K_{X'} - F') \ge p_g(X) - p_g(F) \ge 1.
$$
Since $\dim \Sigma=1$, $\psi$ contracts every element in $|M|$ to points. Note that $h^0(X', M - F') > 0$. Thus $\psi$ contracts general fibers of $f'$ to points. Since $f'$ has connected fibers, we conclude that $\psi$ contracts every fiber of $f'$ by \cite[Lemma 1.6]{Kollar_Mori}. It follows that $\psi = f'$ and the general fiber of $\psi$ is $F'$. If $X$ is irregular, by \cite[Lemma 4.5 (i)]{Chen}, we have $q(X) = 1$ and $h^2(X, \CO_X) = 0$. The proof is completed.
\end{proof}

%\newpage

\section{Noether-Severi inequality for irregular $3$-folds} \label{section: Noether-Severi inequality}

The goal of this section is to prove the Noether-Severi inequality \eqref{eq: main1 inequality}.

\subsection{Noether-Severi inequality for $3$-folds with $(1,2)$-surface Albanese fibers}

Throughout this subsection, we always assume that $X$ is a minimal and irregular $3$-fold of general type with 
$$
f: X \to B
$$ 
the Albanese fibration of $X$ such that the general fiber $F$ of $f$ is a minimal $(1, 2)$-surface. Here $B$ is a smooth curve. By the same argument as in the proof of \cite[Proposition V.15]{Beauville}, we deduce that $B$ is of genus $g(B)=q(X)$. In the following, we fix a sufficiently ample divisor $T$ on $B$ as in Remark \ref{rmk: T} with $\deg T = t$.
%such that the map $\phi_{K_X+f^*T}$ is the relative canonical map $X \dashrightarrow \PP_B(f_* \omega_{X/B})$ of $X$ with respect to $f$. In particular, we have $\dim\Sigma=2$.

\begin{lemma} \label{lem: h0K+T}
We have
$$
h^0(X, K_X + f^*T) = \chi(\omega_X) + 2t - (g(B) - 1).
$$
\end{lemma}

\begin{proof}
By \cite[Theorem 2.1]{Kollar_2}, $R^1f_*\omega_X$ and $R^2f_*\omega_X$ are both torsion free sheaves. Now $q(F)=0$. We deduce that $R^1f_* \omega_X = 0$. Moreover, $R^2f_*\omega_X = \omega_B$. Since $T$ is sufficiently ample, we have
$$
h^1(X, K_X + f^*T) = h^1(B, f_* \omega_X \otimes \CO_B(T)) = 0,
$$
and
$$
h^2(X, K_X + f^*T) = h^0(B, R^2f_*\omega_X\otimes\mathcal{O}_B(T)) = h^0(B, K_B+T) = t + g(B) - 1.
$$
Note that $\chi(\omega_F) = p_g(F) + 1 = 3$. Combine these results together, and it follows that 
\begin{align*}
	h^0(X, K_X + f^*T) & = \chi(\CO_X(K_X + f^*T)) - t - (g(B) - 1) \\
	& = \chi(\omega_X) + 2t - (g(B) - 1).
\end{align*}
Thus the proof is completed.
\end{proof}

\begin{lemma} \label{lem: error}
We have
$$
K_X^3 \ge \frac{4}{3} \chi(\omega_{X})+ \frac{2}{3}(q(X)-1) - \frac{4}{3}.
$$
\end{lemma}

\begin{proof}
Consider the map 
$$
\phi_{K_X+f^*T}: X \dashrightarrow \Sigma \subset \PP^{h^0(X, K_X + f^*T) - 1}
$$ 
induced by $|K_X + f^*T|$. By the assumption on $T$, we have $\dim \Sigma = 2$. Combine Proposition \ref{prop: Noether (1,2)-surface} (2) and Lemma \ref{lem: d} (2) together. It follows that 
$$
K_X^3 \ge \frac{4}{3}h^0(X, K_X + f^*T) - \frac{8}{3} t + 2(g(B) - 1) - \frac{4}{3}.
$$
Then the result just follows from Lemma \ref{lem: h0K+T}.
\end{proof}

\begin{prop} \label{prop: Noether-Severi (1,2)-surface}
We have 
$$
K_X^3 \ge \frac{4}{3} \chi(\omega_X) + \frac{2}{3}(q(X)-1).
$$
\end{prop}

\begin{proof}
Let $\pi_k: B_k \to B$ be any \'etale cover of degree $k > 1$ and let $X_k=X \times_{\pi_k} B_k$. Then we have the following commutative diagram:
$$
\xymatrix{
	X_k \ar[rr] \ar[d]_{f_k} & & X \ar[d]^f \\
	B_k \ar[rr]^{\pi_k} & & B 
}
$$
It is easy to see that the induced fibration $f_k: X_k \to B_k$ via $\pi_k$ is exactly the Albanese fibration of $X_k$, and the Albanese fiber of $X_k$ is a minimal $(1, 2)$-surface. By the Hurwitz formula,  $g(B_k) - 1 = k(g(B)-1)$. By Lemma \ref{lem: error}, we have
$$
K_{X_k}^3 \ge \frac{4}{3} \chi(\omega_{X_k}) + \frac{2}{3}(g(B_k)-1) - \frac{4}{3}.
$$
Since $X_k \to X$ is \'etale of degree $k$, the above inequality is equivalent to 
$$
K_{X}^3 \ge \frac{4}{3} \chi(\omega_{X}) + \frac{2}{3}(g(B)-1) - \frac{4}{3k}.
$$
Thus the proof is completed by taking $k \to \infty$.
\end{proof}

\subsection{Main theorem} Now we prove the main result in this section.

\begin{theorem} \label{thm: Noether-Severi}
Let $X$ be a minimal and irregular $3$-fold of general type. 
Then we have the following optimal inequality:
\begin{equation} \label{eq: Noether-Severi}
	K_X^3 \ge \frac{4}{3} \chi(\omega_{X}).
\end{equation}
If the equality holds, then $q(X)=1$, $h^2(X, \CO_X)=0$, and the general Albanese fiber of $X$ is a minimal $(1,2)$-surface.
\end{theorem}

\begin{proof}
If the general Albanese fiber of $X$ is a $(1,2)$-surface, then \eqref{eq: Noether-Severi} follows from Proposition \ref{prop: Noether-Severi (1,2)-surface}. Otherwise, by \cite[Theorem 1.8]{Hu_Zhang}, we have a stronger inequality $K_X^3 \ge 2 \chi(\omega_{X})$. Therefore, \eqref{eq: Noether-Severi} always holds. By the example constructed in \cite[Section 3]{Hu_Zhang}, \eqref{eq: Noether-Severi} is optimal.

From now on, we assume that $K_X^3 = \frac{4}{3} \chi(\omega_{X})$. Then $K_X^3 < 2\chi(\omega_X)$. By \cite[Theorem 1.8]{Hu_Zhang} again, the general Albanese fiber of $X$ is a $(1, 2)$-surface. Then by Proposition \ref{prop: Noether-Severi (1,2)-surface}, $q(X)=1$.

Since $\chi(\omega_X) = \frac{3}{4}K_X^3 > 0$, we have $\chi(\omega_X)\ge 1$.  Let $\pi'_k: X_k \to X$ be an \'etale cover of $X$ of degree $k \ge 4$. Then we still have $K_{X_k}^3 = \frac{4}{3} \chi(\omega_{X_k})$. By Proposition \ref{prop: Noether-Severi (1,2)-surface} again, $q(X_k) = 1$. Since $\chi(\omega_{X_k}) = k\chi(\omega_{X}) \ge 4$, we deduce that $p_g(X_k) \ge \chi(\omega_{X_k}) - q(X_k) + 1 \ge 4$. Consider the canonical map of $X_k$ (i.e., taking $T = 0$ in \S \ref{section: geometry}). If the canonical image of $X_k$ is a curve, by Proposition \ref{prop: canonical pencil}, $h^2(X_k, \CO_{X_k}) = 0$. If the canonical image is a surface, by Lemma \ref{lem: d} (1), Proposition \ref{prop: Noether (1,2)-surface} (2) and Proposition \ref{prop: shaper Noether (1,2)-surface}, we have
$$
\frac{4}{3} \chi(\omega_{X_k}) = K^3_{X_k} > \frac{4}{3} (p_g(X_k) - 1) = \frac{4}{3} \left(\chi(\omega_{X_k}) + h^2(X_k, \CO_{X_k}) - 1\right).
$$
Thus we still have $h^2(X_k, \CO_{X_k}) = 0$. Since $\mathcal{O}_X$ is a direct summand of ${\pi'_k}_*\mathcal{O}_{X_k}$, we deduce that $h^2(X, \CO_X) = 0$. Thus the proof is completed.
\end{proof}

Theorem \ref{thm: Noether-Severi} has the following consequence.

\begin{coro} \label{coro: Severi type}
Let $X$ be a minimal and irregular $3$-fold of general type. Then we have the following optimal inequality 
$$
K_X^3 \ge \frac{4}{3} h^0_a(X, K_X).
$$
Moreover, the equality holds if and only if $K_X^3 = \frac{4}{3} \chi(\omega_X)$.
\end{coro}

\begin{proof}
When the Albanese dimension of $X$ is not one, by \cite[Corollary B]{Barja}, $K_X^3 \ge 4h^0_a(X, K_X)$. Thus we only need to treat the case when $X$ is of Albanese dimension one.

Let $f: X \to B$ be the Albanese fibration of $X$, where $B$ is a smooth projective curve of genus $g(B) = q(X) \ge 1$. Denote by $F$ a general fiber of $f$. By \cite{Hacon}, $f_* \omega_X$ is a generic vanishing sheaf on $B$. We deduce that
\begin{equation} \label{eq: ha}
	h^0_a(X, K_X) = \chi(f_* \omega_X).
\end{equation}

Suppose first that $F$ is not a $(1, 2)$-surface. By \cite[Theorem 1.6]{Hu_Zhang},
$$
K_{X/B}^3 \ge 2 \deg f_*\omega_{X/B}.
$$
Note that 
$$
K_X^3 = K_{X/B}^3 + 6(g(B) - 1) K_F^2 = K_{X/B}^3 + 6 \chi(\omega_B) K^2_F.
$$ 
By the Riemann-Roch theorem, 
$$
\chi(f_* \omega_X) = \deg f_* \omega_{X} - p_g(F) \chi(\omega_B) = \deg f_* \omega_{X/B} + p_g(F) \chi(\omega_B).
$$
By the Noether inequality, it is easy to check that $3K_F^2 \ge p_g(F)$. Combine the above results with \eqref{eq: ha}. It follows that
$$
K_X^3 \ge 2 \chi(f_*\omega_X) = 2 h^0_a(X, K_X).
$$
Now suppose $F$ is a $(1, 2)$-surface. As is shown in the proof of Lemma \ref{lem: h0K+T},  $R^1f_* \omega_X = 0$ and $R^2f_* \omega_X = \omega_B$. By Theorem \ref{thm: Noether-Severi}, we deduce that
\begin{equation} \label{eq: comparison}
	K_X^3 \ge \frac{4}{3} \chi(\omega_X) = \frac{4}{3} \chi(f_*\omega_X) + \frac{4}{3} \chi(\omega_B) \ge \frac{4}{3} \chi(f_*\omega_X).
\end{equation}
Thus by \eqref{eq: ha}, we always have
\begin{equation} \label{eq: Severi type'}
	K_X^3 \ge \frac{4}{3} h^0_a(X, K_X).
\end{equation}

If the equality in \eqref{eq: Severi type'} holds, then \eqref{eq: comparison} becomes an equality. Thus $K_X^3 = \frac{4}{3} \chi(\omega_X)$. On the other hand, if $K_X^3 = \frac{4}{3} \chi(\omega_X)$, by Theorem \ref{thm: Noether-Severi}, the general Albanese fiber $F$ is a $(1, 2)$-surface and $q(X) = 1$. Thus \eqref{eq: comparison} becomes an equality, so does \eqref{eq: Severi type'}. The proof is completed.
\end{proof}

%\newpage

\section{$3$-folds on the Noether-Severi line: more properties} \label{section: property}

Throughout this section, we assume that $X$ is a minimal and irregular $3$-fold of general type with $K_X^3 = \frac{4}{3} \chi(\omega_{X})$. Let
$$
a: X \to \mathrm{Alb}(X)
$$
be the Albanese morphism of $X$. By Theorem \ref{thm: Noether-Severi}, $B: = \mathrm{Alb}(X)$ is a smooth curve of genus one, and $a$ is the Albanese fibration of $X$. We will also fix a sufficiently ample divisor $T$ on $B$ as in Remark \ref{rmk: T}, and denote $t = \deg T$.

Consider the map $\phi_{K_X+a^*T}$ induced by $|K_X + a^*T|$. We will keep on using the same notation as in \S \ref{subsection: setting (1,2)-surface}. Recall the following commutative diagram:
$$
\xymatrix{
& & X' \ar[d]_{\pi} \ar[lld]_{a'} \ar[rr]^{\psi} \ar[drr]^{\phi_M} & &  \Sigma' \ar[d]^{\tau}  \\
B & & X  \ar@{-->}[rr]_{\phi_{K_X + a^*T}} \ar[ll]^a  & & \Sigma         
}
$$
Here all the notation are the same as in \S \ref{subsection: setting (1,2)-surface}, except that we replace $f$ and $f'$ therein by $a$ and $a'$. Let $F$ be a general fiber of $a$. Since $p_g(F) = 2$, by the choice of $T$, we know that $\phi_{K_X+a^*T}$ is the relative canonical map of $X$ with respect to $a$. Moreover, $\Sigma \simeq \PP_B(a_* \omega_X)$ is a smooth elliptic ruled surface contained in $\PP^{h^0(X, K_X + a^*T) - 1}$. We still write $d = (\deg \tau) \cdot (\deg \Sigma)$.

All the above notation and facts will be used throughout this section. 

\subsection{Relative canonical image of $X$} \label{subsection: canonical image}
We have the following lemma.

\begin{lemma} \label{lem: deg Sigma}
    Keep the notation as above. Then the morphism $\tau: \Sigma' \to \Sigma$ is an isomorphism, and 
    $$
    d = \deg \Sigma = h^0(X, K_X + a^*T) = \chi(\omega_X) + 2t.
    $$
\end{lemma}

\begin{proof}
Since now $g(B) = 1$, we have
$$
K_X^3 \ge \frac{4}{3} d - \frac{8}{3} t \ge \frac{4}{3} h^0(X, K_X + a^*T) - \frac{8}{3} t = \frac{4}{3} \chi(\omega_X).
$$
Here the first inequality is from Proposition \ref{prop: Noether (1,2)-surface} (2), the second is from Lemma \ref{lem: d}, and the third equality is by Lemma \ref{lem: h0K+T}. As $K_X^3 = \frac{4}{3} \chi(\omega_X)$, the above inequalities must be equalities. Thus $\deg \tau = 1$ and $d = \deg \Sigma = h^0(X, K_X + a^*T)$. The proof is completed.
\end{proof}

\subsection{Cartier index of $X$} The main result here is that $X$ is Gorenstein.
\begin{theorem} \label{thm: Cartier}
Let $X$ be a minimal and irregular $3$-fold of general type with $K_X^3 = \frac{4}{3} \chi(\omega_{X})$. Then $X$ is Gorenstein. It follows that $X$ is factorial.
\end{theorem}

\begin{proof}
Suppose $X$ is Gorenstein. Then by \cite[Lemma 5.1]{Kawamata}, $X$ is factorial. Thus to prove the theorem, it suffices to prove that $X$ is Gorenstein. Recall the following Riemann-Roch formula
$$
\chi(\omega_X^{[2]}) = h^0(X, 2K_X) = \frac{1}{2}K_X^3 + 3 \chi(\omega_{X}) + l_2(X) \footnote{In the Riemann-Roch formula, $\omega_X^{[2]}$ means the reflexive sheaf $\mathcal{O}_X(2K_X)$. }
$$
in \cite[Corollary 10.3]{Reid}.  Here the correction term $l_2(X) = 0$ if and only if $X$ is Gorenstein. By the Kawamata-Viehweg vanishing theorem, we have
$$
h^0(X, 2K_X + 2a^*T) = \chi(\CO_X(2K_X + 2a^*T)) = \chi(\omega_X^{[2]}) + 2t \chi(\omega_F^{\otimes 2}).
$$
Since $F$ is a $(1, 2)$-surface, $\chi(\omega_F^{\otimes 2}) = K_F^2 + \chi(\CO_F) = 4$. Thus we deduce that
\begin{equation} \label{eq: relative bicanonical}
	h^0(X, 2K_X+2a^*T) = \frac{1}{2}K_X^3 + 3\chi(\omega_X) + 8t + l_2(X).
\end{equation}
In the following, we prove in steps that $l_2(X) > 0$ would lead to a contradiction. 

\textbf{Step 0}. From now on, suppose $l_2(X) > 0$. For any $k \in \ZZ_{>0}$, let $\mu_k: B \to B$ be the multiplication map by $k$ on $B$, and let $X_k = X \times_{\mu_k} B$. We have
$$
h^0(X_k, 2K_{X_k}) = k^2 h^0(X, 2K_X), \quad K_{X_k}^3 = k^2 K_X^3, \quad \chi(\omega_{X_k}) = k^2 \chi(\omega_X).
$$
This implies that $l_2(X_k) = k^2l_2(X)$. Thus replacing $X$ by $X_k$ for a sufficiently large $k$, we may assume that $l_2(X) \ge 6$. 

Recall that $|M| = \movable|\rounddown{\pi^*(K_X + a^*T)}|$. Denote by $C$ a general fiber of $\psi = \phi_M$. Set $|M_0| = \movable|2K_{X'} + 2 a'^*T|$, $|M_1| = \movable|2K_{X'}+2a'^*T - M|$ and $|M_2|=\movable|2K_{X'} + 2a'^*T - 2M|$. Replacing $X'$ by a further blow-up, we may assume that $|M_0|$, $|M_1|$ and $|M_2|$ are all base point free. By Bertini's theorem, we may take a smooth general member $S \in |M|$. Then $(a'^*T)|_S \equiv tC$. Moreover, it is easy to see that
\begin{align} 
	h^0(X, 2K_X + 2a^*T) & = h^0(X', 2K_{X'} + 2a'^*T) \nonumber \\
	& = u_0 + u_1 + h^0(X', 2K_{X'} + 2a'^*T - 2M), \label{eq: step 0}
\end{align}
where 
$$
u_i = \dim \mathrm{Im} \left(H^0(X', M_i) \to H^0(S, M_i|_S)\right) \quad (i=0, 1).
$$
Note that now we are in the equality case of Proposition \ref{prop: Noether (1,2)-surface} (2). By Proposition \ref{prop: equality} (2), (3) and Lemma \ref{lem: deg Sigma}, we have 
\begin{align} 
	\left((\pi^*K_X)|_S \right)^2 & = \left( \left. \left( \pi^* \left(K_X + \frac{1}{2} a^*T-\frac{1}{2} a^*T \right) \right) \right|_S  \right)^2  \nonumber \\	
	& = \left( \left. \left( \pi^* \left(K_X + \frac{1}{2} a^*T \right) \right) \right|_S  \right)^2 - t\left( (\pi^*K_X) \cdot C \right) \nonumber \\
	& = \frac{2}{3} (2 \deg \Sigma - t) - t \label{eq: step 0'} \\
	& = \frac{4}{3} \chi(\omega_X) + t. \nonumber
\end{align}

\textbf{Step 1}. In this step, we prove that
\begin{equation} \label{eq: step 1}
	u_0 \le \frac{8}{3} \chi(\omega_X)+6t + 2.
\end{equation}

Consider the complete linear system $|M_0|_S|$. By our assumption, it is base point free, thus induces a morphism $\phi_0: S \to \PP^{h^0(S, M_0|_S) - 1}$. If $\dim \phi_0(S) = 2$, by \cite[Lemma 1.8]{Ohno},
$$
4 \left((\pi^*K_X)|_S + t C \right)^2 \ge (M_0|_S)^2 \ge 2h^0(S, M_0|_S) - 4. 
$$
If $\dim \phi_0(S) = 1$, since $M_0|_S \ge M|_S$, the general fiber of $\phi_0$ is identical to $C$. Now $M_0|_S \equiv b C$, where $b \ge h^0(S, M_0|_S) - 1$. Thus by Lemma \ref{lem: E0} (3),
$$
2 \left((\pi^*K_X)|_S + tC\right)^2 \ge b \left((\pi^*K_X) \cdot C \right) \ge  h^0(S, M_0|_S) - 1. 
$$
Thus in both cases, we always have
$$
2 \left((\pi^*K_X)|_S + tC \right)^2 \ge h^0(S, M_0|_S) - 2.
$$
Together with \eqref{eq: step 0'}, it follows that
$$
u_0 \le h^0(S, M_0|_S) \le 2 \left((\pi^*K_X)|_S + tC \right)^2 + 2 = \frac{8}{3} \chi(\omega_X) + 6t + 2.
$$

\textbf{Step 2}. In this step, we prove that
\begin{equation} \label{eq: step 2}
	u_1 \le \chi(\omega_X)+2t + 2.
\end{equation}

Similarly as in Step 1, the complete linear system $|M_1|_S|$ induces a morphism $\phi_1: S \to \PP^{h^0(S, M_1|_S) - 1}$. If $\dim \phi_1(S) = 2$, by \cite[Lemma 1.8]{Ohno}, $(M_1|_S)^2 \ge 2h^0(S, M_1|_S) - 4$. It follows that
$$
4\left((\pi^*K_X)|_S + t C\right)^2 \ge (M_1|_S + M|_S)^2 \ge 2h^0(S, M_1|_S) - 4 + 2(M_1|_S \cdot M|_S).
$$
Note that $g(C)=2$. Since $\phi_1$ does not contract $C$, the linear system $|M_1|_S||_C$ induces a finite morphism on $C$. Note that $g(C) = 2$ by Lemma \ref{lem: g2}. We deduce that $(M_1|_S \cdot C) \ge 2$. Moreover, as in \S \ref{subsection: Noether (1,2)-surface} and by Lemma \ref{lem: deg Sigma}, we have $M|_S \equiv dC = (\chi(\omega_X)+2t) C$. Thus 
$$
(M_1|_S \cdot M|_S) \ge 2 \chi(\omega_X) + 4t .
$$ 
Combine the above inequalities with \eqref{eq: step 0'} together,  we deduce that 
$$
h^0(S, M_1|_S) \le 2 \left((\pi^*K_X)|_S + tC \right)^2 + 2 - 2 \chi(\omega_X) - 4t = \frac{2}{3}\chi(\omega_X)+2t + 2.
$$
Since $\chi(\omega_X) > 0$, we deduce that $u_1 \le \chi(\omega_X) + 2t + 2$.

If $\dim \phi_1(S) = 1$, since $M_1|_S \ge M|_S$, the general fiber of $\phi_1$ is just $C$. Thus $M_1|_S \equiv bC$, where $b \ge h^0(S, M_1|_S) - 1$. Note that $M|_S \equiv (\chi(\omega_X)+2t) C$ as before. Therefore, the divisor $2(\pi^*K_X)|_S - (b + \chi(\omega_X))C$ on $S$ is pseudo-effective. Let $\sigma: S \to S_0$ be the contraction morphism onto the minimal model of $S$. By Proposition \ref{prop: equality} (3), we deduce that $K_{S_0} - (t + b + \chi(\omega_X)) C_0$ is pseudo-effective, where $C_0 = \sigma_* C$. In the meantime, by Proposition \ref{prop: equality} (1) and Lemma \ref{lem: deg Sigma},  $K_{S_0} \equiv (2\chi(\omega_X)+3t) C_0 + 2\Gamma_{S_0}$, where $\Gamma_{S_0}$ is a section of the fibration $S_0 \to B$ with $g(\Gamma_{S_0}) = g(B) = 1$. Then
$$
3\Gamma_{S_0}^2 = \left((K_{S_0} + \Gamma_{S_0}) \cdot \Gamma_{S_0}\right) - \left( 2\chi(\omega_X)+3t \right) (\Gamma_{S_0} \cdot C_0) = -(2\chi(\omega_X)+3t).
$$
This implies that the divisor 
$$
K_{S_0} - \frac{1}{3} \left( 2 \chi(\omega_X)+3t 
\right) C_0 \equiv \frac{2}{3} 
\left( \left( 2\chi(\omega_X)+3t \right) C_0 + 3\Gamma_{S_0} \right)
$$
is nef. Therefore, it follows that
$$
\left(K_{S_0} - \frac{1}{3} \left( 2\chi(\omega_X)+3t \right) C_0 \right) \cdot \left(K_{S_0} - \left( t + b + \chi(\omega_X) \right) C_0 \right) \ge 0,
$$
i.e.,
$$
K_{S_0}^2 \ge \left(\frac{5}{3} \chi(\omega_X) + 2t+b\right) (K_{S_0} \cdot C_0).
$$
Note that $(K_{S_0} \cdot C_0) = 2$ and $K_{S_0}^2 = \frac{16}{3}\chi(\omega_X)+8t$ by Proposition \ref{prop: equality} (2) and Lemma \ref{lem: deg Sigma}. We deduce that $b \le \chi(\omega_X)+2t $. It follows that
$$
u_1 \le h^0(S, M_1|_S) \le b + 1 \le \chi(\omega_X) + 2t + 1.
$$

\textbf{Step 3}. In this step, we prove that 
\begin{equation} \label{eq: step 3}
	h^0(X', 2K_{X'}+2a'^*T - 2M) = 1.
\end{equation}

Suppose on the contrary that $h^0(X', 2K_{X'}+2a'^*T - 2M) \ge 2$. Then $|M_2|$ is base point free. Since $K_X$ is semi-ample and big, we have $((\pi^*K_X)^2 \cdot M_2) > 0$. Otherwise, any effective divisor linear equivalent to $M_2$ would be contracted by the pluricanonical morphism of $X$, which is a contradiction. Thus by \eqref{eq: step 0'}, we deduce that
$$
K_X^3 = \left((\pi^*K_X)^2 \cdot (K_{X'} + a'^*T) \right) - t > \left( (\pi^*K_X)|_S \right)^2 - t = \frac{4}{3} \chi(\omega_X).
$$
This is a contradiction. Thus $h^0(X', 2K_{X'}+2a'^*T - 2M) = 1$.

\textbf{Step 4}. Now we can finish the whole proof. By \eqref{eq: relative bicanonical}, we have
$$
h^0(X, 2K_X+2a^*T) = \frac{1}{2} K_X^3 + 3 \chi(\omega_{X}) + l_2(X)+8t \ge \frac{11}{3} \chi(\omega_X) + 8t + 6.
$$
On the other hand, by \eqref{eq: step 0}, \eqref{eq: step 1}, \eqref{eq: step 2} and \eqref{eq: step 3},
$$
h^0(X, 2K_X+2a^*T) \le \frac{11}{3} \chi(\omega_X)+8t + 5.
$$
This is a contradiction. Thus the whole proof is completed.
\end{proof}

\subsection{Relative canonical linear system of $X$} \label{subsection: K_X}
Recall in \S \ref{subsection: setting (1,2)-surface} that there is a canonical section $\Gamma$ of the fibration $a: X \to B$ whose intersection $\Gamma \cap F$ with $F$ is just the unique base point of $|K_F|$. Since $\Sigma$ is a surface, the restriction morphism
$$
H^0(X, K_X+a^*T) \to H^0(F, K_F)
$$ 
is surjective, which implies that $\Gamma$ is the only horizontal base locus of $|K_X+a^*T|$ with respect to $a$.

As in \S \ref{subsection: setting (1,2)-surface}, we have 
$$
\pi^*(K_X+a^*T) = M + Z, \quad K_{X'} = \pi^*K_X + E_{\pi}.
$$
Let $S \in |M|$ be a general member. Then $S$ is smooth, and $a'|_S: S \to B$ is a fibration with a general fiber $C$. As in \eqref{eq: EVZV}, we may write
$$
E_{\pi}|_S = \Gamma_S + E_V, \quad Z|_S = \Gamma_S + Z_V.
$$
Here $\Gamma_S$ is a section of $a'|_S$ with $\pi(\Gamma_S) = \Gamma$, and $E_V$ and $Z_V$ are the vertical parts of $E_{\pi}|_S$ and $Z|_S$ with respect to $a'|_S$, respectively. Denote by $\sigma: S \to S_0$ the contraction onto the minimal model of $S$. Note that the fibration $a'|_S$ descends to a fibration $S_0 \to B$ with a general fiber $C_0 = {\sigma}_* C$.

\begin{lemma} \label{lem: E effective}
The divisor $E_\pi|_S - Z|_S = E_V - Z_V$ is effective.
\end{lemma}

\begin{proof}
By \eqref{eq: modification} and \eqref{eq: exceptional divisor}, we have
$$
E_\pi|_S - Z|_S = (K_{X'} + S)|_S - \left(\pi^*(2K_X+a^*T)\right)|_S = K_S - \left(\pi^*(2K_X+a^*T) \right)|_S.
$$
Since $K_S \ge \sigma^*K_{S_0}$, by Proposition \ref{prop: equality} (3), $E_V - Z_V = E_\pi|_S - Z|_S$ is $\QQ$-linearly equivalent to an effective $\QQ$-divisor. In particular, there exist an integer $n > 0$ and an effective divisor $D$ on $S$ such that 
$$
nE_V \sim D + nZ_V.
$$ 
On the other hand, Proposition \ref{prop: equality} (1) also tells us that ${\sigma}_* E_V = 0$. It implies that $h^0(S, nE_V) = h^0(S_0, \CO_{S_0}) = 1$. Thus
$$
nE_V = D + nZ_V.
$$
It follows that $E_V - Z_V = \frac{1}{n}D$ is effective.
\end{proof}

\begin{lemma} \label{lem: no fixed part}
The linear system $|K_X+a^*T|$ has no fixed part.
\end{lemma}

\begin{proof}
Suppose $|K_X+a^*T|$ has nonzero fixed part $Z_X$. Then $Z_X$ is vertical respect to $a$. Let $S_X = \pi(S)$. By Theorem \ref{thm: Cartier}, both $S_X$ and $Z_X$ are Cartier divisors on $X$. By the Kawamata-Viehweg vanishing theorem, we have $H^1(X, \CO_X(-S_X-Z_X)) = 0$. As a result, the natural restriction map 
$$
H^0(X, \CO_X) \to H^0(S_X+Z_X, \CO_{S_X+Z_X})
$$ 
is surjective. Thus $h^0(S_X+Z_X, \mathcal{O}_{S_X+Z_X})=1$, which implies that $S_X+Z_X$ is connected. In particular, we have $S_X \cap Z_X \neq \emptyset$. Let $A$ be an integral curve supported on $Z_X|_{S_X}$. By Lemma \ref{lem: E effective}, $E_\pi|_S - (\pi^*Z_X)|_S = (E_\pi|_S - Z|_S) + (Z|_S - (\pi^*Z_X)|_S)$ is effective. In particular, $A \subseteq \mathrm{Supp}(Z_X|_{S_X}) \subseteq \pi(E_\pi)$. We deduce that $A \subseteq \baselocus|S_X|$.

By the construction of $\pi$ in \S \ref{section: modification} and by induction, there is a $\beta$-exceptional prime divisor $E_A$ such that $E_A|_S \neq 0$, $\pi(E_A) = A$ and $\coeff_{E_A}(E_\pi - Z) = \coeff_{E_A}(E_\pi - \pi^*(S_X + Z_X)) < 0$. We deduce that $E_\pi|_S - Z|_S$ is not effective, which contradicts Lemma \ref{lem: E effective}. As a result, $|K_X+a^*T|$ has no fixed part.
\end{proof}

\begin{lemma} \label{lem: movable part normal}
A general member $S_X \in |K_X+a^*T|$ is normal with at worst canonical singularities. Moreover, $\pi|_S: S \to S_X$ factors through $\sigma: S \to S_0$.
\end{lemma}

\begin{proof}
By Bertini's theorem, $S_X$ is integral. Moreover, since $X$ is Gorenstein, $S_X$ is Cohen-Macaulay. To prove that $S_X$ is normal, we only need to show that $S_X$ is smooth in codimension one. 

Suppose the singular locus of $S_X$ contains an integral curve $A$. Then $S_X$ has multiplicity at least two at a general point of $A$. Since $S$ is smooth, similarly to the proof of Lemma \ref{lem: no fixed part}, there is a $\beta$-exceptional divisor $E_A$ on $X'$ such that $\pi(E_A) = A$, $E_A|_S \neq 0$, and $\coeff_{E_A}(E_\pi - Z) < 0$. Thus $E_{\pi}|_S - Z|_S$ is not effective. However, it is impossible by Lemma \ref{lem: E effective}. Therefore, $S_X$ is normal. By the adjunction formula, $K_{S_X} = (2K_X + a^*T)|_{S_X}$. Then
$$
K_S - (\pi|_S)^*K_{S_X} = (K_{X'} + S)|_S - (\pi^* (2K_X + a^*T))|_S = E_\pi|_S - Z|_S,
$$
which is effective by Lemma \ref{lem: E effective}. Thus $S_X$ has at worst canonical singularities. By the uniqueness of the minimal model, the minimal resolution of $S_X$ is just $S_0$. The proof is completed.
\end{proof}

\begin{theorem} \label{thm: |K_X|}
The following statements hold: 
\begin{itemize}
	\item [(1)] $\baselocus|K_X+a^*T| = \Gamma$, and $\Gamma$ lies in the smooth locus of $X$.
	
	\item [(2)] A general member $S_X \in |K_X+a^*T|$ is smooth.
	
	\item [(3)] Let $\pi_\Gamma: Y \to X$ be the blow-up of $X$ along $\Gamma$. Then the linear system $|\pi_\Gamma^*(K_X+a^*T) - E_\Gamma|$ is base point free, where $E_\Gamma$ is the $\pi_\Gamma$-exceptional divisor.
\end{itemize}	
\end{theorem}

\begin{proof}
By Lemma \ref{lem: no fixed part}, $|K_X+a^*T|$ has no fixed part. Thus $\baselocus|K_X+a^*T| = \pi (Z \cap S)$. Thus to prove $\baselocus|K_X+a^*T| = \Gamma$, we only need to prove that $\Gamma = \pi(Z|_S)$. Since $Z|_S = \Gamma_S + Z_V$ and $\pi(\Gamma_S) = \Gamma$, it suffices to show that $\pi(Z_V) \subset \Gamma$.

By Proposition \ref{prop: equality} (1), $\sigma(Z_V)$ consists of finitely many points on $S_0$. Thus by Lemma \ref{lem: movable part normal}, $\pi(Z_V)$ also consists of finitely many points on $\pi(S)$. Suppose there is a point $p \in \pi(Z_V)$ but $p \notin \Gamma$. We may write 
$$
Z_V = Z_1 + Z_2,
$$
where $Z_1$ and $Z_2$ are effective, $\pi(Z_1) = p$, and $p \notin \pi(Z_2)$. By the argument in \S \ref{subsection: Noether (1,2)-surface}, now we have
$$
\left(\pi^*(K_X+a^*T)\right)|_S = M|_S + Z|_S \equiv d C + \Gamma_S + Z_1 + Z_2.
$$
However, since $Z_1$ does not intersect $Z_2$ or $\Gamma_S$, we deduce that 
$$
\left(Z_1 \cdot (dC + \Gamma_S + Z_2) \right) = 0.
$$
This is a contradiction, because $(\pi^*(K_X+a^*T))|_S$ is nef and big, thus $1$-connected. As a result, $\pi(Z_V) \subseteq \Gamma$ and $\baselocus|K_X+a^*T| = \Gamma$.

By Theorem \ref{thm: Cartier}, every Weil divisor on $X$ is Cartier. Take any fiber $F_1$ of $a$. Since $(\Gamma \cdot F_1) = 1$, there exists exactly one irreducible component $F_{1,0}$ of $F_1$ such that $F_{1,0} \cap \Gamma \neq \emptyset$ is a point. Moreover, $(\Gamma \cdot F_{1,0}) = 1$ and $\coeff_{F_{1,0}}(F_1) = 1$. It implies that $F_1$ is smooth at this point, so is $X$. Thus (1) is proved.

To prove (2), suppose $S_X$ is singular. By Bertini's theorem and (1), the singular locus of $S_X$ is contained in $\Gamma$. Let $p \in S_X$ be a singularity. By Lemma \ref{lem: movable part normal}, we have the minimal resolution $\sigma_0: S_0 \to S_X$ such that $K_{S_0} = \sigma_0^*K_{S_X}$. Let $E_p$ be the exceptional divisor on $S_0$ lying over $p$. Since $p$ is a canonical singularity, every irreducible component of $E_p$ is a $(-2)$-curve. In particular, $(K_{S_0} \cdot E_p) = 0$. Let $\Gamma_{S_0} \subset S_0$ be the strict transform of $\Gamma \subset S_X$ under $\sigma_0$. Then we have $(\Gamma_{S_0} \cdot E_p) > 0$. On the other hand, by Proposition \ref{prop: equality} (1), $K_{S_0} \equiv 2\Gamma_{S_0} + (2d - t) C_0$. Thus 
$$
(K_{S_0} \cdot E_p) = 2(\Gamma_{S_0} \cdot E_p) + (2d - t) (C_0 \cdot E_p) \ge 2(\Gamma_{S_0} \cdot E_p). 
$$
This is a contradiction. As a result, $S_X = S_0$ is smooth, and (2) is proved.

To prove (3), first note that $\Gamma$ is smooth. By (1) and (2), we have
$$
K_Y = \pi_\Gamma^*K_X + E_\Gamma, \quad \pi_\Gamma^*(K_X+a^*T) = M_\Gamma + E_\Gamma.
$$
Let $S_Y \in |M_\Gamma|$ be a general member. To prove that $|M_\Gamma|$ is base point free, it suffices to show that the restricted linear system $|M_\Gamma||_{S_Y}$ is base point free.

Suppose $\baselocus(|M_\Gamma||_{S_Y}) \neq \emptyset$. Since $S_X$ is smooth and $S_Y$ is the blow-up of $S_X$ along the curve $\Gamma$, the morphism $\pi_\Gamma|_{S_Y}: S_Y \to S_X$ is actually an isomorphism. Moreover, $(\pi_\Gamma|_{S_Y})^*\Gamma = E_\Gamma|_{S_Y}$. We also have a fibration $b|_{S_Y}: S_Y \to B$, where $b = a \circ \pi_\Gamma$. Now $S_X = S_0$. By Proposition \ref{prop: equality} (1) and (3), we have $(K_X+a^*T)|_{S_X} \sim_\QQ \frac{1}{2}(K_{S_X} + (a^*T)|_{S_X}) \equiv \Gamma + d C$. Note that $\Gamma$ lies in the base locus of $|K_X+a^*T||_{S_X}$. Therefore, we may write $(K_X+a^*T)|_{S_X} = \Gamma + (a|_{S_X})^*D$, where $|D|$ is base point free on $B$. Thus we deduce that 
$$
M_\Gamma|_{S_Y} = (\pi_\Gamma^*(K_X+a^*T))|_{S_Y} - E_\Gamma|_{S_Y} = (b|_{S_Y})^*D.
$$
This implies that $\baselocus(|M_\Gamma||_{S_Y})$ must be vertical curves with respect to $b|_{S_Y}$. On the other hand, note that by (1), $\baselocus|M_\Gamma| \subseteq E_\Gamma$. This implies that $\baselocus(|M_\Gamma||_{S_Y}) \subseteq E_\Gamma|_{S_Y}$ which is horizontal with respect to $b|_{S_Y}$. This is a contradiction. Thus $|M_\Gamma||_{S_Y}$ is base point free, and (3) is proved.
\end{proof}

\subsection{Fundamental group} In this subsection, we compute the topological fundamental group of smooth models of $X$ following the idea of Xiao in \cite{Xiao}.

\begin{theorem} \label{thm: fudamental group}
Let $X_1$ be a smooth model of $X$. Then the topological fundamental group
$$
\pi_1(X_1) \simeq \pi_1(B) \simeq \ZZ^2.
$$
\end{theorem}

\begin{proof}
By the assumption, $B = \mathrm{Alb}(X_1)$. Let 
$$
a_1: X_1 \to B
$$
be the Albanese fibration of $X_1$. Then a general fiber $F_1$ of $a_1$ is a smooth $(1, 2)$-surface. By Theorem \ref{thm: Cartier}, $X$ is factorial. Note that the Albanese fibration $a$ has a section $\Gamma$. We deduce that $a$ does not have multiple fibers. Neither does $a_1$. 

Let $\CV$ be the image of $\pi_1(F_1)$ into $\pi_1(X_1)$. By a similar argument as in the proof of \cite[Lemma 1]{Xiao}, we deduce that $\CV$ is a normal subgroup of $\pi_1(X_1)$. Denote $\CH = \pi_1(X_1)/\CV$. Then we have the following exact sequence
$$
1 \to \CV \to \pi_1(X_1) \to \CH \to 1
$$
as in \cite[\S1]{Xiao}. By \cite[Theorem 4.8]{Horikawa2}, $F_1$ is simply connected.  Thus $\CV = 1$. Let $b_1$, $\ldots$, $b_n$ be all points on $B$ lying under singular fibers of $a_1$, and let $B_0 = B \backslash \{b_1, \ldots, b_n\}$. Using exactly the same proof as \cite[Lemma 2]{Xiao} for the fibration $a_1$ together with the fact that $a_1$ has no multiple fibers, we deduce that $\CH$ is the quotient of $\pi_1(B_0)$ by the normal subgroup generated by the homotopy classes of $\gamma_1, \ldots, \gamma_n$ in $\pi_1(B_0)$, where each $\gamma_i$ is a small loop in $B$ around $b_i$. Note that this quotient is exactly $\pi_1(B)$. Thus $\pi_1(X_1) \simeq \pi_1(B) \simeq \ZZ^2$, and the proof is completed.
\end{proof}

%\newpage

\section{$3$-folds on the Noether-Severi line: fine classification} \label{section: classification}

In the last section, we give an explicit description of the canonical model of irregular $3$-folds on the Noether-Severi line.

Let $X$ be a minimal and irregular $3$-fold of general type with $K_X^3 = \frac{4}{3} \chi(\omega_{X})$. Let $a: X \to B$ be the Albanese fibration of $X$. By Theorem \ref{thm: Noether-Severi}, $B$ is a smooth curve of genus $g(B) = q(X) = 1$. Let $T$ be a sufficiently ample divisor on $B$ as in Remark \ref{rmk: T}. Denote $t = \deg T$.

Let 
$$
\epsilon: X \to X_0 := X_{\rc}
$$
be the contraction from $X$ onto its canonical model $X_{\rc}$. Let $\Gamma_0 = \epsilon(\Gamma)$, where $\Gamma$ is the section of $a$ as is described in \S \ref{subsection: setting (1,2)-surface}. Thus $\Gamma_0$ is a section of the Albanese fibration $a_0: X_0 \to B$ of $X_0$. In particular, $\Gamma_0$ is smooth.

\begin{prop} \label{prop: |K_Xc|}
We have
\begin{itemize}
	\item [(1)] $\baselocus|K_{X_0}+a_0^*T| = \Gamma_0$, and $\Gamma_0$ lies in the smooth locus of $X_{0}$.
	
	\item [(2)] Let $\pi_0: X_1 \to X_0$ be the blow-up along $\Gamma_0$ with $E_0$ the exceptional divisor. Then $E_0$ is a $\PP^1$-bundle over $\Gamma_0$, and $|M_0|:=|\pi_0^*(K_{X_0}+a_0^*T) - E_0|$ is base point free.
\end{itemize}
\end{prop}

\begin{proof}
For (1), note that $K_X+a^*T = \epsilon^*(K_{X_0}+a_0^*T)$. By Theorem \ref{thm: |K_X|} (1), $\epsilon^{-1}(\baselocus|K_{X_0}+a_0^*T|) = \baselocus|K_X+a^*T| = \Gamma$. Thus $\baselocus|K_{X_0}+a_0^*T| = \Gamma_0$. Since $\epsilon|_\Gamma: \Gamma \to \Gamma_0$ is an isomorphism and $\epsilon^{-1}(\Gamma_0)=\Gamma$, by Zariski's main theorem, $\epsilon$ is an isomorphism over a neighbourhood of $\Gamma_0$. By Theorem \ref{thm: |K_X|} (1) again, we deduce that $\Gamma_0$ lies in the smooth locus of $X_0$. For (2), note that $\Gamma_0$ is smooth. Thus $E_0$ is a $\PP^1$-bundle over $\Gamma_0$. The rest part of (2) just follows from Theorem \ref{thm: |K_X|} (3).
\end{proof}

\begin{defi} \label{def: section}
We call $\Gamma_0$ the \emph{base-locus section} of $a_0$.
\end{defi}

By Proposition \ref{prop: |K_Xc|}, we have the following commutative diagram
$$
\xymatrix{
X_1 \ar[d]_{\pi_0} \ar[rr]^{\psi} & & \Sigma \ar[d]^{p}\\
X_0 \ar@{-->}[urr]^{\phi} \ar[rr]_{a_0}  & & B       
}
$$
where $\phi := \phi_{K_{X_0} + a_0^*T}$ is the relative canonical map of $X_0$ with respect to $a_0$, $\Sigma = \PP_B((a_0)_* \omega_{X_0})$ is a $\PP^1$-bundle over $B$ with the natural projection $p$, and $\psi$ is the morphism induced by $|M_0|$ which, by Lemma \ref{lem: deg Sigma}, has connected fibers. 

Denote by $C$ a general fiber of $\psi$. By Lemma \ref{lem: E0} (3),  we have $((\pi_0^*K_{X_0}) \cdot C) = 1$. By Lemma \ref{lem: g2}, $(K_{X_1} \cdot C) = 2g(C)-2 = 2$. Thus we deduce that 
$$
(E_0 \cdot C) = \left( (K_{X_1} - \pi_0^*K_{X_0}) \cdot C \right) = 1.
$$
In particular, $\psi|_{E_0}: E_0 \to \Sigma$ is birational. By Proposition \ref{prop: |K_Xc|} (2), $E_0$ itself is a $\PP^1$-bundle. We deduce that $\psi|_{E_0}$ is an isomorphism.

\begin{lemma} \label{lem: flatness}
The morphism $\psi$ is flat, and every fiber of $\psi$ is integral.
\end{lemma}

\begin{proof}
Let $W$ be any fiber of $\psi$. Since $\psi|_{E_0}$ is an isomorphism, $E_0$ is a section of $\psi$. Thus $W \cap E_0$ is a point. Suppose $W$ contains an irreducible component $W_0$ of dimension two. Then $W_0 \cap E_0 = \emptyset$. Thus
$$
\left((K_{X_0}+a_0^*T)^2 \cdot \left( (\pi_0)_* W_0 \right)\right) = \left((M_0 + E_0)^2 \cdot W_0\right) = (M_0^2 \cdot W_0) =0.
$$
Since $K_{X_0}$ is ample, the above equality implies that $\dim (\pi_0)_* W_0 \le 1$. This is impossible, because the only $\pi_0$-exceptional divisor is $E_0$. As a result, $\dim W = 1$. By \cite[Theorem 23.1]{Matsumura}, $\psi$ is flat. 

Note that $K_{X_0}$ is Cartier and $((\pi_0^*K_{X_0})\cdot W)= 1$.  If $W$ is reducible, then there is an irreducible component $W_1$ of $W$ such that $((\pi_0^*K_{X_0})\cdot W_1) = 0$, i.e., $(K_{X_0} \cdot ((\pi_0)_*W_1)) = 0$. Thus $(\pi_0)_*W_1$ is a point, so $W_1 \subset E_0$. This is a contradiction. Therefore, $W$ is irreducible.	Since $K_{X_0}$ is Cartier, $((\pi_0^*K_{X_0})\cdot W)= 1$ and $W$ is irreducible, we deduce that $W$ is reduced.
\end{proof}

\begin{lemma}\label{lem: rank 2 bundle}
Let $\CE= \psi_* \CO_{X_1}(2E_0)$. Then $\CE$ is locally free of rank two.
\end{lemma}

\begin{proof}
Take any fiber $C$ of $\psi$. Since a general fiber of $\psi$ is of genus $2$, by Lemma \ref{lem: flatness}, it follows that $C$ is an integral curve of arithmetic genus $2$. We deduce that $h^1(C,\CO_C)=2$. By Theorem \ref{thm: Cartier} and Proposition \ref{prop: |K_Xc|}, $X_1$ is Cohen-Macaulay and the dualizing sheaf $\omega_{X_1}$ is invertible. Since $C$ is a fiber of $\psi$ and $\Sigma$ is smooth, $C$ is Cohen-Macaulay and the dualizing sheaf $\omega_C=\omega_{X_1}|_C$ is invertible. We deduce that 
$$
h^0(C, K_C)=h^0(C, \omega_C)=h^1(C, \mathcal{O}_C)=2,
$$
where the last equality follows from the Serre duality. On the other hand, note that $K_{X_1} = \pi_0^*K_{X_0} + E_0 = M_0 - \pi_0^*(a_0^*T) + 2E_0$. Thus $h^0(C, 2E_0|_C) = h^0(C, K_{X_1}|_C) =h^0(C, K_C)= 2$. By Grauert's theorem \cite[III, Corollary 12.9]{Hartshorne}, the result follows.
\end{proof}

Let $q: Y:= \PP_{\Sigma}(\CE) \to \Sigma$ denote the $\PP^1$-bundle over $\Sigma$. By Lemma \ref{lem: flatness}, every fiber $C$ of $\psi$ is integral. By \cite[Theorem 3.3]{Catanese_Franciosi_Hulek_Reid}, we deduce that $|K_C| = |2E_0|_C|$ is base point free. Thus we obtain a morphism
$$
f: X_1 \to Y,
$$
which is just the relative canonical map of $X_1$ with respect to $\psi$. Moreover, $f$ is a finite morphism of degree two. By \cite[Theorem 23.1]{Matsumura}, $f$ is flat. Let $E_Y  = f(E_0)$. It is easy to see that $E_Y$ is a section of $q$. Let $j: \Sigma \to Y$ denote this section. Thus we have the following commutative diagram
$$
\xymatrix{
X_1 \ar[drr]_{\psi}  \ar[rr]^{f} & &  Y \ar[d]_q & & \\
& & \Sigma \ar[rr]^p \ar@/_1 pc/[u]_j & & B
}
$$

For any fiber $C$ of $\psi$, $f|_C: C\to f(C)\cong\PP^1$ is just the canonical map of $C$. Note that $2E_0|_C\in |K_C|$. Thus we have $(f^*E_Y)|_C=2E_0|_C$. We conclude that $f^*E_Y=2E_0$. Since $f^*E_Y = 2E_0$, we have $\CE = q_* \CO_Y(E_Y)$. Consider the short exact sequence
$$
0 \to \CO_Y \to \CO_Y(E_Y) \to \CO_{E_Y}(E_Y) \to 0.
$$
Note that $R^1 q_* \CO_Y = 0$ and $q_*\CO_{E_Y}(E_Y) = \CO_\Sigma (j^*E_Y)$. Pushing forward by $q$ on the above short exact sequence, we obtain
\begin{equation} \label{eq: ses}
0 \to \CO_{\Sigma} \to \CE \to \CO_\Sigma (j^*E_Y) \to 0.
\end{equation}

In the following, we will compute $j^*E_Y$. Now $a_1:=p \circ \psi: X_1 \to B$ is the Albanese fibration of $X_1$. Then 
$$
p_0: = a_1|_{E_0}: E_0 \to B
$$ 
induces the $\PP^1$-bundle structure on $E_0$. Denote $(a_0)_* \omega_{X_0}$ by $\CE_B$. By the isomorphism $\psi|_{E_0}$, we may identify $E_0$ with $\PP_B(\CE_B)$. Denote by $F_1$ a general fiber of $a_1$. Then $F_1|_{E_0}$ is just a fiber of $p_0$. Since $M_0 \cdot F_1 \equiv C$, we have $(M_0|_{E_0} \cdot F_1|_{E_0}) = (M_0 \cdot F_1 \cdot E_0) = (E_0 \cdot C) = 1$. Thus we may write
\begin{equation} \label{eq: ME}
M_0|_{E_0} = s + p_0^*A_1,
\end{equation}
where $s$ is a relative hyperplane section of $p_0$, i.e., $\CO_{E_0}(s) = \CO_{E_0}(1)$. Note that $E_0$ is horizontal with respect to $\psi$ but $M_0$ is vertical. We deduce that $h^0(X_1, M_0 - E_0) = 0$. Therefore, the map $\psi|_{E_0}: E_0 \to \Sigma$ is just induced by $|M_0|_{E_0}|$. By Lemma \ref{lem: deg Sigma}, we have
\begin{equation} \label{eq: A and e}
\deg \CE_B + 2 \deg A_1 = (s + p_0^*A_1)^2 = \deg \Sigma = \chi(\omega_X) + 2t.
\end{equation}

Recall that now the inequality in Proposition \ref{prop: Noether (1,2)-surface} (1) becomes an equality. Thus we deduce from Lemma \ref{lem: deg Sigma} that $(K_{X_0} \cdot \Gamma_0) = (K_X \cdot \Gamma) = \frac{1}{3} \chi(\omega_X)$. Let $\gamma_0: B \to X_0$ represent the base-locus section $\Gamma_0$. Then we have
\begin{equation} \label{eq: KE}
(\pi_0^*K_{X_0})|_{E_0} = p_0^*A_0,
\end{equation}
where $A_0 = \gamma_0^* K_{X_0} $ with $\deg A_0 = \frac{1}{3} \chi(\omega_X)$. Set $A_2 = A_0 + T - A_1$. Then we have 
\begin{align}\label{eq: A2}
\deg A_2 =  \frac{1}{3} \chi(\omega_X) + t - \deg A_1
\end{align}
and
\begin{equation} \label{eq: EE}
E_0|_{E_0} = (\pi_0^*K_{X_0} + a_1^*T - M_0)|_{E_0}= - s + p_0^*A_2.
\end{equation}

\begin{remark} \label{rmk: degree 0}
In fact, from  \eqref{eq: A and e} and \eqref{eq: A2}, we deduce that
$$
\deg A_1 + 3 \deg A_2 - \deg \CE_B - t = 0.
$$
\end{remark}

From now on, we identify $E_0$ and $p_0$ with $\Sigma$ and $p$ under the isomorphism $\psi|_{E_0}$. Under this identification and by \eqref{eq: EE}, we have
\begin{equation} \label{eq: E_Y}
j^*E_Y = (\psi|_{E_0})_* \left((f|_{E_0})^*E_Y\right) = (\psi|_{E_0})_* \left(2E_0|_{E_0}\right) = -2s + 2p^*A_2.
\end{equation} 
Thus it follows from \eqref{eq: ses} that
\begin{align} 
K_Y & = -2E_Y + q^*(K_\Sigma - 2s + 2p^*A_2) \nonumber \\
& = -2E_Y + q^*(-4s + p^* \det \CE_B + 2p^* A_2). \label{eq: K_Y}
\end{align}

Let $D$ be the branch locus of $f: X_1 \to Y$. Since $X_1$ has at worst canonical singularities, $D$ must be reduced. Since $E_Y$ is contained in $D$, we may write
$$
D = E_Y + D',
$$ 
where $D'$ does not contain $E_Y$ as an irreducible component. Furthermore, there exists a divisor $L$ on $Y$ such that $D \sim 2L$ and $K_{X_1} = f^*(K_Y + L)$.

\begin{lemma} \label{lem: branch locus}
We have 
\begin{itemize}
	\item [(1)] $L \sim 3E_Y + q^*(5s + p^*(A_1 - 2A_2 - \det \CE_B - T))$;
	\item [(2)] $D' \sim 5E_Y + 2q^*(5s + p^*(A_1 - 2A_2 - \det \CE_B - T))$.
\end{itemize}
\end{lemma}

\begin{proof}
Recall that by \eqref{eq: ME}, we have
$$
K_{X_1}  = M_0 - a_1^* T + 2E_0 = f^*\left( E_Y + q^*\left( s + p^*(A_1-T) \right) \right).
$$
Let $L' = 3E_Y + q^*(5s + p^*(A_1 - 2A_2 - \det \CE_B - T))$. By \eqref{eq: K_Y}, we have
$$
f^*L' = K_{X_1} - f^*K_Y = f^*L,
$$
i.e., $f^*\CO_Y(L-L') = \CO_{X_1}$. By the projection formula, we deduce that
$$
\CO_Y \oplus \CO_Y(-L) = f_* \CO_{X_1} = \CO_Y(L-L') \oplus \CO_Y(-L'). 
$$
Thus $L \sim L'$. Thus (1) is proved. Since $D' \sim 2L - E_Y$, (2) follows immediately from (1).
\end{proof}

\begin{lemma} \label{lem: D'}
Under the isomorphism $q|_{E_Y}: E_Y \to \Sigma$, we have 
$$
\CO_{E_Y}(D') = \CO_{E_Y} \left(2p^*(A_1 +  3A_2 - \det \CE_B - T) \right) = \CO_{E_Y}.
$$
\end{lemma}

\begin{proof}
By \eqref{eq: E_Y} and Lemma \ref{lem: branch locus} (2), we have
\begin{align*}
	D'|_{E_Y} & \sim 5E_Y|_{E_Y} + 2 \left(5s + p^*(A_1 - 2A_2 - \det \CE_B - T)\right) \\ 
	& \sim 2p^*(A_1 + 3 A_2 - \det \CE_B - T).
\end{align*}
By Remark \ref{rmk: degree 0}, $\deg A_1 + 3 \deg A_2 - \deg \CE_B - t = 0$.
Since $h^0(E_Y, D'|_{E_Y}) > 0$, we deduce that $2A_1 + 6 A_2 - 2 \det \CE_B - 2T \sim 0$. Thus the result follows.
\end{proof}

\begin{lemma} \label{lem: split}
The short exact sequence \eqref{eq: ses} splits. In particular,
$$
\CE = \CO_\Sigma \oplus \CO_\Sigma (-2s + 2p^*A_2).
$$
\end{lemma}

\begin{proof}
Let $h: (\Sym^4 \CE) \otimes \CE \to \Sym^5 \CE$ be the symmetrizing morphism. Then we have the following commutative diagram:
$$
\xymatrix{
	0 \ar[r] & \Sym^4 \CE \ar[r] \ar@{=}[d] & (\Sym^4 \CE) \otimes \CE \ar[r] \ar[d]^h &  (\Sym^4 \CE) \otimes \CO_\Sigma (j^*E_Y) \ar[r] \ar[d] & 0 \quad (*)\\
	0 \ar[r] & \Sym^4 \CE \ar[r] & \Sym^5 \CE \ar[r] & \CO_\Sigma (5j^*E_Y) \ar[r] & 0 \quad (**)  
}
$$
Here $(*)$ is obtained by tensoring \eqref{eq: ses} with $\Sym^4 \CE$, and $(**)$ is obtained by pushing forward by $q$ on 
$$
0 \to \CO_Y(4E_Y) \to \CO_Y(5E_Y) \to \CO_{E_Y}(5E_Y) \to 0
$$
as well as $R^1 q_*\CO_Y(4E_Y) = 0$. The splitting of \eqref{eq: ses} is equivalent to the splitting of $(*)$. Thus to prove the lemma, it suffices to prove that $(**)$ splits, because by composing $h$, a splitting morphism $\Sym^5 \CE \to \Sym^4 \CE$ gives a splitting morphism $(\Sym^4 \CE) \otimes \CE \to \Sym^4 \CE$.

Suppose $(**)$ does not split. Then the extension class 
$$
[\Sym^5 \CE] \in \mathrm{Ext}^1 (\CO_\Sigma (5j^*E_Y), \Sym^4 \CE) = \mathrm{Ext}^1 (\CO_\Sigma, \Sym^4 \CE \otimes \CO_\Sigma (-5j^*E_Y))
$$ 
is nonzero. By \cite[III. Ex. 6.1]{Hartshorne}, the map 
$$
H^0(\Sigma, \CO_\Sigma) \to H^1 (\Sigma, \Sym^4 \CE \otimes \CO_\Sigma (-5j^*E_Y))
$$ 
is nonzero, thus injective. From $(**)$, we deduce that
$$
h^0(\Sigma, \Sym^4 \CE \otimes \CO_\Sigma (-5j^*E_Y)) = h^0(\Sigma, \Sym^5 \CE \otimes \CO_\Sigma (-5j^*E_Y)).
$$
Together with the projection formula and \eqref{eq: E_Y}, we have 
$$
h^0(Y, 4E_Y + 10 q^*(s - p^*A_2)) = h^0(Y, 5E_Y + 10 q^*(s - p^*A_2)).
$$
In particular, $E_Y \subseteq \baselocus|5E_Y + 10 q^*(s - p^*A_2)|$. On the other hand, by Lemma \ref{lem: branch locus} (2) and Lemma \ref{lem: D'}, 
$$
5E_Y + 10 q^*(s - p^*A_2) \sim 5E_Y + 2q^*(5s + p^*(A_1 - 2A_2 - \det \CE_B - T)) \sim D'.
$$
Thus $E_Y$ is an irreducible component of $D'$. This is a contradiction. As a result, $(**)$ splits, so does \eqref{eq: ses}. 
\end{proof}

\begin{lemma} \label{lem: A=T}
We have $A_1 = T$. Thus $A_2 = A_0$.
\end{lemma}

\begin{proof}
Recall that $K_{X_1} = f^*(E_Y + q^*(s + p^*(A_1 - T)))$. By the projection formula, we have
$$
\psi_* \omega_{X_1} = \CO_\Sigma (s + p^*(A_1 - T)) \otimes q_* \left( \CO_Y (E_Y) \oplus \CO_Y(E_Y - L) \right).
$$
By Lemma \ref{lem: branch locus} (1), it is clear that $q_*\CO_Y(E_Y - L) = 0$. Thus by Lemma \ref{lem: split}, we have
\begin{align*}
	\psi_* \omega_{X_1} & = \CO_\Sigma (s + p^*(A_1 - T)) \otimes q_* \CO_Y (E_Y) \\
	& = \CO_\Sigma (s + p^*(A_1 - T)) \oplus \CO_\Sigma (-s + p^*(A_1 + 2A_2 - T)).
\end{align*}
Therefore, it follows from the projection formula again that
$$
(a_1)_* \omega_{X_1} = p_* (\psi_* \omega_{X_1}) = p_* \CO_\Sigma (s + p^*(A_1 - T)) = (a_0)_* \omega_{X_0} \otimes \CO_B(A_1 - T).
$$
As a result, we deduce that $A_1 = T$, and thus $A_2 = A_0 + T - A_1 = A_0$.
\end{proof}

Now we are ready to state the main theorem of this section. 

\begin{theorem} \label{thm: classification}
Let $X_0$ be a canonical and irregular $3$-fold of general type with $K_{X_0}^3 = \frac{4}{3} \chi(\omega_{X_0})$. Let $a_0: X_0 \to B$ be the Albanese fibration of $X_0$, where $B$ is a smooth curve of genus one. Let $X_1$ be the blow-up of $X_0$ along the base-locus section $\Gamma_0$ of $a_0$. Then the Albanese fibration $a_1: X_1 \to B$ of $X_1$ is factorized as
$$
a_1: X_1 \stackrel{f} \longrightarrow Y \stackrel{q} \longrightarrow \Sigma \stackrel{p} \longrightarrow B
$$
with the following properties:
\begin{itemize}
	\item [(1)] $\Sigma = \PP_B(\CE_B)$ with $p: \Sigma \to B$ the projection, where $\CE_B = (a_0)_* \omega_{X_0}$.
	
	\item [(2)] $Y = \PP_\Sigma (\CO_\Sigma \oplus (\CO_\Sigma (-2) \otimes \CK_1^2))$ with $q: Y \to \Sigma$ the projection, where $\gamma_0: B \to X_0$ corresponds to the section $\Gamma_0$ and $\CK_1 = p^*(\gamma_0^*\omega_{X_0})$.
	
	\item [(3)] $f: X_1 \to Y$ is a flat double cover with the branch locus 
	$$
	D = D_1 + D_2,
	$$ 
	where $D_1 \in |\CO_Y(1)|$, $D_2 \in |\CO_Y(5) \otimes q^* (\CO_\Sigma (10) \otimes \CK_1^{-4} \otimes \CK_2^{-2})|$, and $D_1 \cap D_2 = \emptyset$.	Here $\CK_2 = p^*(\det \CE_B)$.
\end{itemize}
\end{theorem}

\begin{proof}
The statement (1) is from the definition. By Lemma \ref{lem: A=T}, $A_2 = A_0$. Thus (2) just follows from Lemma \ref{lem: split}. Now $\CO_Y(1) = \CO_Y(E_Y)$. Moreover, by Lemma \ref{lem: A=T},
$$
5s + p^*(A_1 - 2A_2 - \det \CE_B - T)  = 5s - p^*(2A_0 + \det \CE_B).
$$
Thus (3) follows from Lemma \ref{lem: branch locus} (2).
\end{proof}

%\newpage

\bibliography{Ref_NoetherSeveri3fold}

\providecommand{\bysame}{\leavevmode\hbox to3em{\hrulefill}\thinspace}
\providecommand{\MR}{\relax\ifhmode\unskip\space\fi MR }
\providecommand{\MRhref}[2]{%
  \href{http://www.ams.org/mathscinet-getitem?mr=#1}{#2}
}
\providecommand{\href}[2]{#2}
\begin{thebibliography}{10}

\bibitem{Barja}
Miguel~A. Barja, \emph{Generalized {C}lifford-{S}everi inequality and the
  volume of irregular varieties}, Duke Math. J. \textbf{164} (2015), no.~3,
  541--568. \MR{3314480}

\bibitem{Barja_Pardini_Stoppino_Severi_Line}
Miguel~\'{A}ngel Barja, Rita Pardini, and Lidia Stoppino, \emph{Surfaces on the
  {S}everi line}, J. Math. Pures Appl. (9) \textbf{105} (2016), no.~5,
  734--743. \MR{3479190}

\bibitem{Barja_Pardini_Stoppino}
\bysame, \emph{Higher-dimensional {C}lifford--{S}everi equalities}, Commun.
  Contemp. Math. \textbf{22} (2020), no.~8, 1950079, 15. \MR{4142335}

\bibitem{Barja_Pardini_Stoppino2}
\bysame, \emph{Linear systems on irregular varieties}, J. Inst. Math. Jussieu
  \textbf{19} (2020), no.~6, 2087--2125. \MR{4167003}

\bibitem{Beauville}
Arnaud Beauville, \emph{Complex algebraic surfaces}, London Mathematical
  Society Student Texts, vol.~34, Cambridge University Press, Cambridge, 1996,
  Translated from the 1978 French original by R. Barlow, with assistance from
  N. I. Shepherd-Barron and M. Reid. Second edition. \MR{1406314}

\bibitem{Beltrametti_Sommese}
Mauro~C. Beltrametti and Andrew~J. Sommese, \emph{The adjunction theory of
  complex projective varieties}, De Gruyter Expositions in Mathematics,
  vol.~16, Walter de Gruyter \& Co., Berlin, 1995. \MR{1318687}

\bibitem{Bombieri}
E.~Bombieri, \emph{Canonical models of surfaces of general type}, Inst. Hautes
  \'{E}tudes Sci. Publ. Math. (1973), no.~42, 171--219. \MR{318163}

\bibitem{Catanese_Franciosi_Hulek_Reid}
Fabrizio Catanese, Marco Franciosi, Klaus Hulek, and Miles Reid,
  \emph{Embeddings of curves and surfaces}, Nagoya Math. J. \textbf{154}
  (1999), 185--220. \MR{1689180}

\bibitem{Chen_Chen_Jiang}
Jungkai~A. Chen, Meng Chen, and Chen Jiang, \emph{The {N}oether inequality for
  algebraic 3-folds}, Duke Math. J. \textbf{169} (2020), no.~9, 1603--1645,
  With an appendix by J\'{a}nos Koll\'{a}r. \MR{4105534}

\bibitem{Chen_Debarre_Jiang}
Jungkai~Alfred Chen, Olivier Debarre, and Zhi Jiang, \emph{Varieties with
  vanishing holomorphic {E}uler characteristic}, J. Reine Angew. Math.
  \textbf{691} (2014), 203--227. \MR{3213551}

\bibitem{Chen}
Meng Chen, \emph{Inequalities of {N}oether type for 3-folds of general type},
  J. Math. Soc. Japan \textbf{56} (2004), no.~4, 1131--1155. \MR{2092941}

\bibitem{Enriques}
Federigo Enriques, \emph{Le {S}uperficie {A}lgebriche}, Nicola Zanichelli,
  Bologna, 1949. \MR{0031770}

\bibitem{Hacon}
Christopher~D. Hacon, \emph{A derived category approach to generic vanishing},
  J. Reine Angew. Math. \textbf{575} (2004), 173--187. \MR{2097552}

\bibitem{Hartshorne}
Robin Hartshorne, \emph{Algebraic geometry}, Springer-Verlag, New
  York-Heidelberg, 1977, Graduate Texts in Mathematics, No. 52. \MR{0463157}

\bibitem{Horikawa1}
Eiji Horikawa, \emph{Algebraic surfaces of general type with small
  {$c^{2}_{1}.$} {I}}, Ann. of Math. (2) \textbf{104} (1976), no.~2, 357--387.
  \MR{424831}

\bibitem{Horikawa2}
\bysame, \emph{Algebraic surfaces of general type with small {$c^{2}_{1}$}.
  {II}}, Invent. Math. \textbf{37} (1976), no.~2, 121--155. \MR{460340}

\bibitem{Horikawa5}
\bysame, \emph{Algebraic surfaces of general type with small {$c^{2}_{1}$}.
  {V}}, J. Fac. Sci. Univ. Tokyo Sect. IA Math. \textbf{28} (1981), no.~3,
  745--755 (1982). \MR{656051}

\bibitem{Hu}
Yong Hu, \emph{Inequality for {G}orenstein minimal 3-folds of general type},
  Comm. Anal. Geom. \textbf{26} (2018), no.~2, 347--359. \MR{3805162}

\bibitem{Hu_Zhang_Noether}
Yong Hu and Tong Zhang, \emph{Algebraic threefolds of general type with small
  volume}, preprint.

\bibitem{Hu_Zhang}
\bysame, \emph{Fibered varieties over curves with low slope and sharp bounds in
  dimension three}, J. Algebraic Geom. \textbf{30} (2021), no.~1, 57--95.
  \MR{4233178}

\bibitem{Jiang}
Zhi Jiang, \emph{On {S}everi type inequalities}, Math. Ann. \textbf{379}
  (2021), no.~1-2, 133--158. \MR{4211084}

\bibitem{Kawamata}
Yujiro Kawamata, \emph{Crepant blowing-up of 3-dimensional canonical
  singularities and its application to degenerations of surfaces}, Ann. of
  Math. \textbf{127} (1988), no.~2, 93--163. \MR{924674}

\bibitem{Kollar_2}
J\'anos Koll\'ar, \emph{Higher direct images of dualizing sheaves. i.}, vol.
  123, 1986, pp.~11--42. \MR{825838}

\bibitem{Kollar_Mori}
J\'anos Koll\'ar and Shigefumi Mori, \emph{Birational geometry of algebraic
  varieties}, Cambridge Tracts in Mathematics, vol. 134, Cambridge University
  Press, Cambridge, 1998, With the collaboration of C. H. Clemens and A. Corti,
  Translated from the 1998 Japanese original. \MR{1658959}

\bibitem{Lu_Zuo}
Xin Lu and Kang Zuo, \emph{On {S}everi type inequalities for irregular
  surfaces}, Int. Math. Res. Not. IMRN (2019), no.~1, 231--248. \MR{3897429}

\bibitem{Matsumura}
Hideyuki Matsumura, \emph{Commutative ring theory}, Cambridge Studies in
  Advanced Mathematics, vol.~8, Cambridge University Press, Cambridge, 1989,
  Translated from the Japanese by M. Reid. Second edition. \MR{1011461}

\bibitem{Nagata}
Masayoshi Nagata, \emph{On rational surfaces. {I}. {I}rreducible curves of
  arithmetic genus {$0$} or {$1$}}, Mem. Coll. Sci. Univ. Kyoto Ser. A. Math.
  \textbf{32} (1960), 351--370. \MR{126443}

\bibitem{Noether}
Max Noether, \emph{Zur {T}heorie des eindeutigen {E}ntsprechens algebraischer
  {G}ebilde}, Math. Ann. \textbf{8} (1875), no.~2, 495--533. \MR{1509663}

\bibitem{Ohno}
Koji Ohno, \emph{Some inequalities for minimal fibrations of surfaces of
  general type over curves}, J. Math. Soc. Japan \textbf{44} (1992), no.~4,
  643--666. \MR{1180441}

\bibitem{Pardini}
Rita Pardini, \emph{The {S}everi inequality {$K^2\geq 4\chi$} for surfaces of
  maximal {A}lbanese dimension}, Invent. Math. \textbf{159} (2005), no.~3,
  669--672. \MR{2125737}

\bibitem{Reid}
Miles Reid, \emph{Young person's guide to canonical singularities}, Algebraic
  geometry, {B}owdoin, 1985 ({B}runswick, {M}aine, 1985), Proc. Sympos. Pure
  Math., vol.~46, Amer. Math. Soc., Providence, RI, 1987, pp.~345--414.
  \MR{927963}

\bibitem{Severi}
Francesco Severi, \emph{La serie canonica e la teoria delle serie principali di
  gruppi di punti sopra una superficie algebrica}, Comment. Math. Helv.
  \textbf{4} (1932), no.~1, 268--326. \MR{1509461}

\bibitem{Xiao}
Gang Xiao, \emph{{$\pi_1$} of elliptic and hyperelliptic surfaces}, Internat.
  J. Math. \textbf{2} (1991), no.~5, 599--615. \MR{1124285}

\bibitem{Zhang}
Tong Zhang, \emph{Severi inequality for varieties of maximal {A}lbanese
  dimension}, Math. Ann. \textbf{359} (2014), no.~3-4, 1097--1114. \MR{3231026}

\end{thebibliography}
\bibliographystyle{amsplain}

\end{document}